\documentclass[a4paper,10pt]{article}
\usepackage[utf8]{inputenc}
\usepackage{amsmath,amssymb,amsthm}
\usepackage{empheq}
\usepackage{graphicx,enumerate,color}
\usepackage{hyperref}

\newtheorem{thm}{Theorem}[section]
\newtheorem{lem}[thm]{Lemma}
\newtheorem{prop}[thm]{Proposition}
\newtheorem{cor}[thm]{Corollary}
\newtheorem{rem}[thm]{Remark}

\newcommand{\NN}{\mathbb{N}}
\newcommand{\ZZ}{\mathbb{Z}}
\newcommand{\RR}{\mathbb{R}}

\newcommand{\eps}{\varepsilon}
\newcommand{\diam}{\text{diam}}

\begin{document}

\title{Upper bounds for Courant-sharp Neumann and Robin eigenvalues}

\author{K. Gittins\footnote{Department of Mathematical Sciences, Durham University, Mathematical Sciences and Computer Science Building, Upper Mountjoy Campus, Stockton Road, Durham, DH1 3LE, United Kingdom. \tt{katie.gittins@durham.ac.uk}}, C. L\'ena\footnote{Università degli Studi di Padova, Dipartimento di Tecnica e Gestione dei Sistemi Industriali (DTG), Stradella S. Nicola 3, 36100 Vicenza and  Dipartimento di Matematica ``Tullio Levi-Civita", via Trieste 63, 35121 Padova, Italy. \tt{corentin.lena@unipd.it}}}

\date{\today}

\maketitle

\begin{abstract}
We consider the eigenvalues of the Laplacian on an open, bounded, connected set in $\RR^n$ with $C^2$ boundary, with a Neumann boundary condition or a Robin boundary condition. We obtain upper bounds for those eigenvalues that have a corresponding eigenfunction which achieves equality in Courant's Nodal Domain theorem. In the case where the set is also assumed to be convex, we obtain explicit upper bounds in terms of some of the geometric quantities of the set.
~\\

\emph{Corrigendum.} A previous version of this work was accepted and published by the \emph{Bulletin de la Société Mathématique de France} (see \cite{GL2020} in the bibliography of Appendix \ref{app:corr}). It contained a gap: the classical (Euclidean) Faber-Krahn inequality was applied in a setting where it might not hold. This version reproduces the previous one with the addition of a corrigendum in Appendix \ref{app:corr} that addresses the issue. All the results in Sections 2--8 and most of those in Section 9 are thus preserved.
\end{abstract}

\paragraph{MSC classification (2010):}  	35P15, 49R05, 35P05.


\section{Introduction}

\subsection{Statement of the problem}

Let $\Omega$ be an open, bounded, connected set in $\RR^n$, $n \geq 2$, with Lipschitz boundary $\partial \Omega$.
Consider the Neumann Laplacian acting on $L^2(\Omega)$ and note that it has discrete spectrum
since $\Omega$ is bounded. The Neumann eigenvalues of $\Omega$ can hence be written in a non-decreasing
sequence, counted with multiplicity,
\begin{equation*}
0=\mu_1(\Omega) < \mu_2(\Omega) \leq \dots \leq \mu_k(\Omega) \leq \dots,
\end{equation*}
where the only accumulation point is $+\infty$.

By Courant's Nodal Domain theorem, any eigenfunction corresponding to $\mu_k(\Omega)$ has at most $k$ nodal domains.
If $u_k$ is an eigenfunction corresponding to $\mu_k(\Omega)$ with $k$ nodal domains, then we call it a
Courant-sharp eigenfunction. In this case, we also call $\mu_k(\Omega)$ a Courant-sharp eigenvalue of $\Omega$.

The Courant-sharp property was first considered by Pleijel \cite{Pl} in 1956 for the Dirichlet Laplacian.
In particular, Pleijel proved that there are only finitely many Courant-sharp Dirichlet eigenvalues of
a bounded, planar domain with sufficiently regular boundary. See \cite{BM82, Pe57} for generalisations
of Pleijel's theorem to higher dimensions and other geometric settings. Following from Pleijel's result,
natural questions are, for a given domain, how many such eigenvalues are there and how large are they?

The recent articles \cite{BH16,vdBkG16cs} consider these questions and give upper bounds for the largest
Courant-sharp Dirichlet eigenvalue and the number of such eigenvalues in terms of some of the geometric
quantities of the underlying domain.
In Theorem 1.3 of \cite{BH16}, the authors bound this number using the area, perimeter, maximal curvature and minimal cut-distance to the boundary, for a set in $\RR^2$ which is sufficiently regular but not necessarily convex
(the cut-distance will be defined in Section~\ref{Sec3}).
In \cite{vdBkG16cs}, such geometric upper bounds are obtained for an open set in $\RR^n$ with finite Lebesgue measure.
In the case where the domain is convex, the upper bound given in Example 1 of \cite{vdBkG16cs} could be expressed in terms of the isoperimetric ratio of the domain. From this, one can deduce that if the domain has a large number of Courant-sharp Dirichlet eigenvalues then its isoperimetric ratio is also large.

It was shown recently in \cite{cL16} that if $\Omega$ is an open, bounded, connected set in $\RR^n$
with $C^{1,1}$ boundary, then the Neumann Laplacian acting in $L^2(\Omega)$ has finitely many Courant-sharp
eigenvalues (we refer to \cite{cL16} for a description of prior results). As mentioned in \cite{BH16}, the aforementioned questions are also interesting for the Courant-sharp eigenvalues
of the Neumann Laplacian.

\subsection{Goal of the article}

The aim of the present article is to obtain upper bounds for the Courant-sharp Neumann eigenvalues
in the case where $\Omega \subset \RR^n$ is open, bounded, connected with $C^2$ boundary.
In the case where $\Omega$ is also convex, we obtain explicit upper bounds for the Courant-sharp Neumann eigenvalues of $\Omega$, and for the number of such eigenvalues, in terms of some of the geometric quantities of $\Omega$.
These results correspond to some of those mentioned above for the Dirichlet case, with some
additional hypotheses due to the difficulties in handling the Neumann boundary condition.

We follow the same strategy that was used in \cite{cL16}. This involves distinguishing
between the nodal domains of a Courant-sharp eigenfunction $u$ for which the majority of the $L^2$ norm
of $u$ comes from the interior (bulk domains) and those for which the majority of the $L^2$ norm of $u$
comes from near the boundary (boundary domains), and then obtaining upper bounds for the number of each type
of nodal domain. In the first case, the argument used by Pleijel \cite{Pl}, which rests upon
the Faber-Krahn inequality, can be used as the eigenfunction in a bulk domain almost satisfies a
Dirichlet boundary condition.
For the boundary domains, it is not possible to employ the same argument as Pleijel as these nodal domains have mixed Dirichlet-Neumann boundary conditions so the Faber-Krahn inequality cannot be employed.
The strategy of \cite{cL16} to deal with the boundary domains is to locally straighten the boundary of the
domain $\Omega$ and then to reflect the nodal domain in order to obtain a new domain that almost satisfies
a Dirichlet boundary condition. One then has to compare the $L^2$ norm of the gradient of an eigenfunction
corresponding to a Courant-sharp eigenvalue on the boundary domain to the $L^2$ norm of the gradient of
the reflected eigenfunction on the reflected domain. See Section~\ref{secEstimates}.

We restrict our attention to Euclidean domains with $C^2$ boundary. We can then
make use of tubular coordinates in order to set up and describe the reflection procedure explicitly.
This allows us to keep explicit control of the constants appearing in the aforementioned estimates in order
to obtain estimates for the Courant-sharp Neumann eigenvalues.

In Proposition~\ref{propmubound}, we obtain an upper bound for the Courant-sharp Neumann eigenvalues of $\Omega$ in terms of some of its geometric quantities. More specifically, it depends on $|\Omega|$ the area of $\Omega$, $\rho(\Omega)$ the isoperimetric ratio to the power $1/2$, $ t_{+}(\Omega)$ the smallest radius of curvature of the boundary, and the cut distance to the boundary (see Section~\ref{Sec3} for precise definitions of the latter quantities).

A simpler presentation of this upper bound is possible in the case where $\Omega$ is convex, since one of the additional conditions in the general case is no longer required (see Section~\ref{Sec8}). In addition, we obtain an upper bound for the number of such eigenvalues by using the upper bound for the Neumann counting function which is proved in Appendix~\ref{sec:A}. In particular, we have the following proposition.

\begin{prop}\label{prop:1.2}
Let $\Omega$ be an open, bounded, convex set in $\RR^2$ with $C^2$ boundary.
There exist constants $C>0$ and $C'>0$, that do not depend on $\Omega$, such that for any Courant-sharp eigenvalue $\mu_k(\Omega)$,
\begin{equation}\label{eq:p1.2.1}	
	\mu_k(\Omega)\le C\left(\frac{|\Omega|}{t_+(\Omega)^4}+\frac{\rho(\Omega)^8}{|\Omega|}\right)
\end{equation}
and
\begin{equation}\label{eq:p1.2.2}
	k\le C'\left(\frac{|\Omega|^2}{t_+(\Omega)^4}+\rho(\Omega)^{8}\right).
\end{equation}
\end{prop}
We note that the left-hand side and right-hand side of Inequality \eqref{eq:p1.2.1} have the same homogeneity
with respect to scaling, and that Inequality \eqref{eq:p1.2.2} is scaling invariant.
In addition, in Section~\ref{Sec8}, we obtain an upper bound for the Courant-sharp Neumann eigenvalues which also
depends upon the diameter of $\Omega$.

By Proposition~\ref{prop:1.2}, we then observe that if $\Omega$ is a sufficiently regular convex set with a large number of Courant-sharp eigenvalues, it has a large isoperimetric ratio or a large curvature at some point of its boundary (or both).
If we additionally assume that $\mu_k(\Omega)$ is large compared with $|\Omega|t_+(\Omega)^{-4}$, we can conclude that the isoperimetric ratio is large. We note that a large isoperimetric ratio is enough to generate a large number of Courant-sharp Neumann eigenvalues. Indeed, this is the case for a rectangle $(0,1) \times (0,L)$ with $L$ large.
By contrast, to the best of the authors' knowledge, it is not known whether a boundary point with large curvature alone can generate many Courant-sharp eigenvalues. It could be interesting to investigate this further.

By $-\Delta_{\Omega}^{\beta}$, we denote the Laplacian on $\Omega$ with the following Robin boundary condition
\begin{equation*}
\frac{\partial u}{\partial \nu} + \beta u = 0 \text{ on $\partial \Omega$, }
\end{equation*}
where $\frac{\partial u}{\partial \nu}$ is the exterior normal derivative and $\beta:\partial\Omega\to\RR$ is a non-negative, Lipschitz continuous function.
We denote the corresponding eigenvalues by $(\mu_k(\Omega, \beta))_{k\ge1}$.
It was shown in \cite{cL16} that there are finitely many Courant-sharp eigenvalues of $-\Delta_{\Omega}^{\beta}$.
By monotonicity of the Robin eigenvalues with respect to $\beta$, we obtain the same results for the
Courant-sharp Robin eigenvalues (see Subsection~\ref{ss:2.2}).

In addition, we obtain analogous results to those mentioned above for any dimension $n \geq 3$, namely Propositions \ref{propBoundGenNDim}, \ref{propConvNDim} and  \ref{propConvNDimSimple} in Section~\ref{SecNdim}.

\subsection{Organisation of the article}

In Section~\ref{Sec2}, we show that in order to obtain upper bounds for the largest Courant-sharp eigenvalue
$\mu$, it is sufficient to obtain upper bounds for the number of nodal domains and the remainder of the
Dirichlet counting function. Estimates for the latter are obtained in Section~\ref{secRemainder}.
To deal with the former, we first consider the 2-dimensional case and set up tubular coordinates in Section~\ref{Sec3}.
Following \cite{cL16}, we then define cut-off functions in Section~\ref{secCutOff} that allow us to distinguish between
bulk and boundary domains. In Subsection~\ref{ssecStraightening} we perform the straightening of the boundary procedure and obtain the desired estimates. We then use these estimates in  Subsection~\ref{ssecNoofnodaldomains} to obtain an
explicit upper bound for the number of Courant-sharp eigenvalues. In Subsection~\ref{ssecgeom}  and Subsection~\ref{ss:5.4}, by taking the geometry
of the domain into account, we improve the estimates from Subsection~\ref{ssecStraightening} in special cases.
We then combine all of the preceding results in Section~\ref{Sec7} to obtain an upper bound for the largest Courant-sharp eigenvalue. In Section~\ref{Sec8}, we obtain explicit upper bounds for the largest Courant-sharp eigenvalue and the number of Courant-sharp eigenvalues of an open and convex planar domain with $C^2$ boundary that involve some of its geometric quantities. In particular, we prove Proposition~\ref{prop:1.2}. In Section~\ref{SecNdim}, we obtain analogous results in arbitrary dimension $n \geq 3$. In Appendix~\ref{sec:A}, we prove an upper bound for the Neumann counting function of a convex set, which is used in the two preceding sections to control the number of Courant-sharp eigenvalues.

\section{Preliminaries}\label{Sec2}

\subsection{Strategy for Courant-sharp Neumann eigenvalues}\label{ss:2.1}

For $\mu>0$, we define the Neumann counting function as follows:
\begin{equation*}
	N_\Omega^{N}(\mu):=\sharp\{k\in\NN\,:\,\mu_k(\Omega)<\mu\}.
\end{equation*}

Let $(\lambda_k(\Omega))_{k\ge1}$ denote the Dirichlet eigenvalues of the Laplacian on $\Omega$. By the min-max characterisations of the Neumann and Dirichlet eigenvalues, we have, for $k \in \NN$, that
\begin{equation}\label{eq:monDN}
\mu_k(\Omega) \leq \lambda_k(\Omega).
\end{equation}
For $\mu>0$, we define the Dirichlet counting function:
\begin{equation*}
	N_\Omega^{D}(\mu):=\sharp\{k\in\NN\,:\,\lambda_k(\Omega)<\mu\},
\end{equation*}
and the corresponding \emph{remainder} $R_{\Omega}^{D}(\mu)$ such that
\begin{equation}
\label{eqDefRemainder}
	N_{\Omega}^{D}(\mu)=\frac{\omega_n \vert \Omega \vert}{(2\pi)^{n}}\mu^{n/2}-R_{\Omega}^{D}(\mu),
\end{equation}
where $\omega_n$ denotes the Lebesgue measure of the ball of radius $1$ in $\RR^n$, and the first term in the right-hand side of Equation \ref{eqDefRemainder} corresponds to Weyl's law.
By \eqref{eq:monDN}, we have
\begin{equation*}
N_{\Omega}^{N}(\mu) \geq N_{\Omega}^{D}(\Omega),
\end{equation*}
and therefore
\begin{equation*}
	N_{\Omega}^{N}(\mu) \geq \frac{\omega_n \vert \Omega \vert}{(2\pi)^{n}}\mu^{n/2}-R_\Omega^D(\mu).
\end{equation*}

Consider an eigenpair $(\mu,u)$ for the Neumann Laplacian, and denote by $\nu(u)$ the number of its nodal domains.
If $u$ is a Courant-sharp eigenfunction associated with $\mu>0$, $\mu=\mu_k(\Omega)$ with $\nu(u)=k$.
On the other hand, Courant's Nodal Domain theorem implies that $\mu_{k-1}(\Omega)<\mu_k(\Omega)$,
so that $N^N_\Omega(\mu)=k-1$. We therefore have
\begin{equation}
\label{eqBasicNC}
 N^N_\Omega(\mu)-\nu(u)<0.
\end{equation}

Hence, in order to obtain upper bounds for $\mu$, we require upper bounds for $\nu(u)$ and $R_\Omega^D(\mu)$.
These will be obtained in Sections~\ref{secEstimates}, \ref{secRemainder} respectively.

To obtain an upper bound for $\nu(u)$, we follow the strategy of \cite{cL16} in which
an important step is to straighten the boundary locally. By restricting our attention to domains with $C^2$ boundary,
we can make use of tubular coordinates to straighten the boundary which allow us to obtain the desired explicit
estimates in Subsection~\ref{ssecStraightening}.

\subsection{Application to Courant-sharp Robin eigenvalues}\label{ss:2.2}
Analogous arguments to the above hold for the Robin eigenvalues $(\mu_k(\Omega,\beta))_{k\ge1}$.
For $\mu>0$, we define the Robin counting function:
\begin{equation*}
	N_\Omega^{\beta}(\mu):=\sharp\{k\in\NN \,:\,\mu_k(\Omega,\beta)<\mu\}.
\end{equation*}
For $k \in \NN$, the Robin eigenvalues satisfy the following monotonicity property with respect to $\beta \geq 0$:
\begin{equation*}
\mu_k(\Omega) \leq \mu_k(\Omega, \beta) \leq \lambda_k(\Omega),
\end{equation*}
where $\beta \equiv 0$ gives rise to the Neumann eigenvalues and $\beta \to +\infty$ corresponds to the Dirichlet eigenvalues. We then have
\begin{equation*}
N_{\Omega}^{\beta}(\mu) \geq N_{\Omega}^{D}(\Omega),
\end{equation*}
and hence
\begin{equation*}
N_{\Omega}^{\beta}(\mu) \geq \frac{\omega_n \vert \Omega \vert}{(2\pi)^{n}}\mu^{n/2}-R_\Omega^D(\mu).
\end{equation*}
(See, for example, \cite[Section 4]{cL16}).

{\section{Tubular coordinates in 2D}\label{Sec3}}

For any $r>0$, we define the inner tubular neighbourhood of $\partial \Omega$ with radius $r$:
\begin{equation*}
	\partial\Omega_{r}^+:=\{x\in\Omega\,;\,\mbox{dist}(x,\partial\Omega)<r\},
\end{equation*}
and its volume
\begin{equation*}
	\tau(r):=|\partial\Omega_{r}^+|.
\end{equation*}
Let us now assume that $\Omega$ is simply connected  (we consider multiply connected domains at the end of Section \ref{secEstimates}).  Let $\gamma: [0,L]\to \RR^2$ be a closed, simple and $C^2$ curve parametrised by arc length, such that $\partial\Omega=\gamma([0,L])$. In particular, $L$ is the total length of $\partial\Omega$. For each $s\in [0,L]$, we write $\mathbf t(s):=\gamma'(s)$, the unit tangent vector at $\gamma(s)$, and denote by $\mathbf n(s)$ the unit vector such that $(\mathbf t(s), \mathbf n(s))$ is a direct orthonormal basis. Up to reversing the orientation of the curve $\gamma$, we can assume that $\mathbf n(s)$ points towards the interior of $\Omega$. The signed curvature of $\gamma$ at the point $\gamma(s)$ (relative to our choice of orientation), denoted by $\kappa(s)$, is then defined by
\begin{equation*}
	\mathbf t'(s)=\kappa(s)\,\mathbf n(s).
\end{equation*}

We define the mapping
\begin{equation}
\label{eqF}
	\begin{array}{cccc}
		F:&[0,L]\times \RR&\to    &\RR^2\\
		  &q:=(s,t)          &\mapsto&x:=\gamma(s)+t\,\mathbf n(s).	
	\end{array}
\end{equation}
The function $F$ is of class $C^1$, and its differential at $q:=(s,t)$, expressed from the base $((1,0),(0,1))$ to the base $(\mathbf t(s),\mathbf  n(s))$, is
\begin{equation}
\label{eqJacobi}
J(q)=\left(\begin{array}{cc}
		1-t\,\kappa(s)&0\\
		0             &1
	  \end{array}
\right).
\end{equation}
Following Subsection 3.1 of \cite{BH16}, we define
\begin{equation}
\label{eqCriticalt}
	t_+(\Omega):=\left(\sup_{s\in[0,L]}\left|\kappa(s)\right|\right)^{-1},
\end{equation}
and, for $s\in[0,L]$, the (internal) \emph{cut-distance} to $\partial\Omega$ at $\gamma(s)$:
\begin{equation}
\label{eqLocCut}
	\delta_+(s):=\sup\{\delta>0\,;\,\mbox{dist}(F(s,t),\partial\Omega)=t\mbox{ for all }t\in[0,\delta]\}.
\end{equation}
We set
\begin{equation}
\label{eqCut}
 \underline\delta_+(\Omega):=\inf_{s\in[0,L]}\delta_+(s)
\end{equation}
and
\begin{equation}
\label{eqDelta0}
	\delta_0(\Omega):=\min\{t_+(\Omega),\,\underline\delta_+(\Omega) \}.
\end{equation}

By construction, $F$ is a diffeomorphism of class $C^1$ from $(0,L)\times(0,\delta_0(\Omega))$ to $\partial\Omega_{\delta_0(\Omega)}^+\setminus \ell_0$, where $\ell_0$ is the segment  $F\left(\{0\}\times(0,\delta_0(\Omega))\right)$.

Given $f: \RR\to\RR$ continuous and piecewise $C^1$ and $\delta\le\delta_0(\Omega)$, we define the function $\varphi: \partial \Omega_{\delta}^+\setminus\ell_0\to \RR$ by $\varphi(x)=f(t)$, where $x=F(s,t)$. By definition of $\delta_0(\Omega)$, this can alternatively be written as $\varphi(x)=f\left(\mbox{dist}\left(x,\partial\Omega\right)\right)$.

\begin{prop} The function $\varphi$ is continuous, and is of class $C^1$ except on the regular arcs
\begin{equation*}
\Gamma_i =\{x\in \Omega\,;\,\mbox{dist}(x,\partial\Omega)=t_i\},
\end{equation*} where $\{t_i\,;\,1\le i\le N\}$ are the points of discontinuity of $f'$ in $(0,\delta_0(\Omega)]$. Furthermore,
\begin{equation}
\label{eqGrad}
	\nabla\varphi(x)=f'(t)\mathbf n(s),
\end{equation}
with $x=F(s,t)$, if $x\notin\bigcup_{i=1}^N\Gamma_i$.
\end{prop}
\begin{proof} This follows from the chain rule, using the expression for the Jacobian matrix given in Equation \ref{eqJacobi}.
\end{proof}

\section{Cut-off functions}
\label{secCutOff}

The purpose of this section is to define cut-off functions $\varphi_0^{\delta}, \varphi_1^{\delta}$
in order to characterise the nodal domains as bulk domains or boundary domains, as in Subsection 2.2 of \cite{cL16}.
The key point here is to obtain explicit estimates.

\vspace{10pt}

{\bf Step 1:} We construct two functions $\chi_0,\,\chi_1:\RR\to\RR$ which are continuous, piecewise $C^1$, and satisfy
\begin{enumerate}[i.]
	\item for $t\le \frac14$, $\chi_0(t)=0$ and $\chi_1(t)=1$;
	\item for $t\in \left[\frac14,\frac34\right]$, $0\le \chi_0(t),\,\chi_1(t)\le 1$;
	\item for $t \ge \frac34$, $\chi_0(t)=1$ and $\chi_1(t)=0$;
	\item $\chi_0^2+\chi_1^2=1$;
\end{enumerate}
We write $B:=\max(\|\chi_0'\|_{L^\infty},\|\chi_1'\|_{L^\infty})$.

We first construct $\psi:\RR\to \RR$, continuous, piecewise $C^1$ and non-decreasing, such that $\psi(0)=0$ and $\psi(1)=1$ and satisfying $0\le\psi\le1$. We ask for the additional condition
\begin{equation*}
	\psi(t)^2+\psi(1-t)^2=1
\end{equation*}
for all $t\in \RR$. One possible choice for $\psi$ is given at the end of this section.

We then set
\begin{align*}
	\chi_0(t)&:=\psi\left(2\left(t-\frac14\right)\right);\\
	\chi_1(t)&:=\psi\left(2\left(\frac34-t\right)\right).
\end{align*}
The functions $\chi_0$ and $\chi_1$ have the desired properties with $B=2\|\psi'\|_{L^\infty}$.

\vspace{10pt}

{\bf Step 2:} For each $\delta\in(0,\delta_0(\Omega)]$, we construct two functions $\varphi_0^{\delta},\,\varphi_1^{\delta}:\Omega\to \RR$ which are continuous and piecewise $C^1$ (the gradient is continuous except for finite jumps on regular arcs), and satisfy, for some positive constant $C$ independent of $\delta$,
\begin{enumerate}[i.]
	\item for $\mbox{dist}\left(x,\partial \Omega\right)\le \frac{\delta}4$, $\varphi_0^{\delta}(x)=0$ and $\varphi_1^{\delta}(x)=1$;
	\item for $\mbox{dist}\left(x,\partial \Omega\right)\in \left[\frac{\delta}4,\,\frac{3\,\delta}4\right]$, $0\le\varphi_0^{\delta}(x),\,\varphi_1^{\delta}(x)\le 1$;
	\item for $\mbox{dist}\left(x,\partial \Omega\right)\ge \frac{3\,\delta}4$, $\varphi_0^{\delta}(x)=1$ and $\varphi_1^{\delta}(x)=0$;
	\item $\left(\varphi_0^{\delta}\right)^2+\left(\varphi_1^{\delta}\right)^2=1$;
	\item $\left\|\nabla\varphi_0^{\delta}\right\|_{L^{\infty}},\,\left\|\nabla\varphi_1^{\delta}\right\|_{L^{\infty}}\le C\delta^{-1}$.
\end{enumerate}

The functions $\chi_0$ and $\chi_1$ of Step 1 being given, we define, for $x\in \partial\Omega_{\delta}^+$,
\begin{align*}
	\varphi_0^{\delta}(x)&:=\chi_0\left(\frac{t}{\delta}\right);\\
	\varphi_1^{\delta}(x)&:=\chi_1\left(\frac{t}{\delta}\right).
\end{align*}
We make the obvious extensions of $\varphi_0^{\delta}$ and $\varphi_1^{\delta}$ so that both functions are continuous. From the enumerated properties of $\chi_0$ and $\chi_1$ and Equation \eqref{eqGrad}, it follows that  $\varphi_0^{\delta}$ and $\varphi_1^{\delta}$ have the desired properties with $C=B$.

\vspace{10pt}

{\bf Explicit constants:}
In order to obtain explicit estimates in what follows, it is necessary to specify $\psi$.
One possible choice for $\psi$ is given by
\begin{equation*}
	\psi(t)^2:=\frac1{\int_0^1s(1-s)\,ds}\int_0^t s(1-s)\,ds,
\end{equation*}
for $t\in[0,1]$, extended by $0$ for $t\notin[0,1]$. More explicitly,
\begin{equation*}
	\psi(t)=\sqrt{3t^2-2t^3}
\end{equation*}
 for $t\in[0,1]$. In that case $\|\psi'\|_{L^\infty}=\sqrt3$. The functions $\chi_0$ and $\chi_1$ that we construct satisfy the listed properties with $B=2\sqrt3$, and the functions $\varphi_0^{\delta}$ and $\varphi_1^{\delta}$ with $C=2\sqrt3$.

\section{Estimates of the nodal count}
\label{secEstimates}

We wish to count the number of each type of nodal domain.
For the bulk domains, one considers a Pleijel-type argument via the Faber-Krahn inequality.
For the boundary domains, one reflects them in the boundary (after straightening it) and then applies the Faber-Krahn inequality to the reflected domains.
See Section 2 of \cite{cL16} for the full details and Subsection~\ref{ssecNoofnodaldomains} below.

\subsection{Straightening of the boundary}
\label{ssecStraightening}

We start by giving explicit versions of some estimates in Subsection 2.4 of Reference \cite{cL16}, in order to control some quantities of interest when we straighten the boundary using tubular coordinates.

\begin{lem}\label{lemEstInt}
Let $V$ be an open set
$V\subset (0,L)\times(0,3\,\delta_0(\Omega)/4)$ and let $U=F(V)$. There exist constants $0<m_-\le m_+$ such that, for any measurable and non-negative function $g:V\to \RR$,
\begin{equation*}
	m_-\int_Vg\,dq\le \int_U f\,dx \le m_+\int_Vg\,dq,
\end{equation*}
where $f=g\circ F^{-1}$. Furthermore, we can choose $m_-=1/4$ and $m_+=7/4$.
\end{lem}

\begin{proof} This follows directly from the change of variable formula
\begin{equation*}
	\int_Uf(x)\,dx=\int_Vg(q)\left(1-t\,\kappa(s)\right)\,dq,
\end{equation*}
and the fact that $V\subset (0,L)\times(0,3\,\delta_0(\Omega)/4)$, so that $t|\kappa(s)|\le 3/4$.
\end{proof}

By taking $g=1$, we deduce the following from Lemma \ref{lemEstInt}.

\begin{cor} \label{corVol}
If $V$ is an open set $V\subset (0,L)\times(0,3\,\delta_0(\Omega)/4)$ and $U=F(V)$, we have
\begin{equation}
\label{eqArea}
	m_-|V|\le|U|\le m_+|V|,
\end{equation}
where $m_+$, $m_-$ are the constants in Lemma \ref{lemEstInt}.
\end{cor}

\begin{cor}
\label{corVolTube}
There exists a positive constant $M$ such that, for all $r\in(0,3\delta_0(\Omega)/4]$,
\begin{equation}
\label{eqVolTube}
	\tau(r)\le MLr.
\end{equation}
Furthermore, we can choose  $M=m_+$, where $m_+$ is the constant in Lemma \ref{lemEstInt}.
\end{cor}

\begin{prop} \label{propIneqQuot}
There exists a positive constant $K$ such that the following holds: if $V$ is an open set
$V\subset (0,L)\times(0,3\,\delta_0(\Omega)/4)$, $U=F(V)$, $u\in H^1(U)$ and $v:=u\circ F$, we have
\begin{equation}
\label{eqIneqQuot}
	\frac{\int_V\left|\nabla v\right|^2\,dq}{\int_Vv^2\,dq}\le K\,\frac{\int_U\left|\nabla u\right|^2\,dx}{\int_Uu^2\,dx}.
\end{equation}
Furthermore, we can choose $K=4m_+$, where $m_+$ is the constant in Lemma \ref{lemEstInt}.
\end{prop}

\begin{proof} From Lemma \ref{lemEstInt}, we have directly
\begin{equation*}
	\int_Uu^2\,dx\le m_+\int_Vv^2\,dq.
\end{equation*}
On the other hand, for any $q\in V$,
\begin{equation*}
	\nabla v(q)=J(q)^T\left(\nabla u\right)\circ F(q),
\end{equation*}
so that
\begin{equation*}
	\left|\nabla u\right|^2\circ F(q)=\left|\left(J(q)^T\right)^{-1}\nabla v(q)\right|^2.
\end{equation*}
From this we deduce, using the change of variable $x=F(q)$, that
\begin{multline*}
	\int_U\left|\nabla u\right|^2\,dx=\int_V\left(\frac1{1-t\,\kappa(s)}\left(\partial_s v(q)\right)^2+\left(1-t\,\kappa(s)\right)\left(\partial_t v(q)\right)^2\right)\,dq\ge\\
	\frac14\int_V \left|\nabla v\right|^2\,dq.
\end{multline*}
Putting together the two previous estimates, we obtain Inequality \eqref{eqIneqQuot}.
\end{proof}

\subsection{Number of nodal domains}
\label{ssecNoofnodaldomains}
To ease the notation in the computations that follow, we set $A:=|\Omega|$ in this section. Let us fix $\eps_0 \in (0,1)$. We consider an eigenpair $(\mu,u)$ for the Neumann Laplacian (as in Subsection 2.2 of \cite{cL16}) or for $-\Delta_{\Omega}^{\beta}$ (as in Section 4 of \cite{cL16})).
We define $u_0:=\varphi_0^\delta u$ and $u_1:=\varphi_1^\delta u$, so that
\begin{equation*}
	u^2=u_0^2+u_1^2.
\end{equation*}
Following Subsection 2.1 of \cite{cL16}, we say that a nodal domain $D$ of $u$ is a \emph{bulk domain} when
\begin{equation*}
\int_Du_0^2\,dx\ge (1-\eps_0)\int_Du^2\,dx
\end{equation*}
and a \emph{boundary domain} when
\begin{equation*}
\int_Du_1^2\,dx>\eps_0\int_Du^2\,dx.
\end{equation*}
We denote by $\nu_0(\eps_0,u)$ and $\nu_1(\eps_0,u)$ the number of bulk and boundary domains respectively, so that
\begin{equation*}
\nu(u)=\nu_0(\eps_0,u)+\nu_1(\eps_0,u).
\end{equation*}
We now need to bound $\nu_0(\eps_0,u)$ and $\nu_1(\eps_0,u)$ from above. Reproducing the argument in Subsection 2.3 of \cite{cL16} with $n=2$, we obtain
\begin{equation}
\label{eqIneqBulk}
	\nu_0(\eps_0,u)\le \frac{A}{\Lambda}\left(\frac{1+\eps_0}{1-\eps_0}\mu+\left(\frac{1+\frac{1}{\eps_0}}{1-\eps_0}\right)\frac{C^2}{\delta^2}\right),
\end{equation}
where $\Lambda$ is the first eigenvalue of the Dirichlet Laplacian on the disc of unit area and $C$ is the constant given in the construction of $\varphi_0^\delta$ and $\varphi_1^\delta$ in Section \ref{secCutOff}.

Let us now consider a boundary domain $D^1_j$  ($1\le j\le \nu_1(\eps_0,u)$). In a similar way to  Subsection 2.4 of \cite{cL16}, we define
\begin{equation*}
	U:=\{x\in D^1_j\,;\, u_1(x)\neq0\}\subset D^1_j\cap\partial\Omega_{\frac{3\delta}4}^+.
\end{equation*}
To simplify notation, we set $\tilde u:=u_1$. We consider $V=F^{-1}(U)$ and $\tilde v=\tilde u\circ F$. We  define
\begin{equation*}
	\sigma(s,t):=(s,-t)
\end{equation*}
and $V^R$ as the interior of the closure of $V\cup\sigma(V)$. We extend $\tilde v$ to a function $\tilde v^R$ on $V^R$ such that $\tilde v^R\circ \sigma=\tilde v^R$ (i.e. $\tilde v^R$ is even with respect to the variable $t$).
By the Faber-Krahn inequality,
\begin{equation*}
	\frac{\Lambda}{\left|V^R\right|}\le \frac{\int_{V^R}\left|\nabla \tilde v^R\right|^2\,dq}{\int_{V^R}\left|\tilde v^R\right|^2\,dq},
\end{equation*}
and therefore
\begin{equation*}
	\frac{\Lambda}{2\,\left|V\right|}\le \frac{\int_{V}\left|\nabla \tilde v\right|^2\,dq}{\int_{V}\left|\tilde v\right|^2\,dq}.
\end{equation*}

Applying Inequalities \eqref{eqArea} and \eqref{eqIneqQuot}, we obtain
\begin{equation}
\label{eqIneqFKBound}
	\frac{m_-\Lambda}{2\,\left|U\right|}\le K\frac{\int_{U}\left|\nabla \tilde u\right|^2\,dx}{\int_{U}\left|\tilde u\right|^2\,dx}.
\end{equation}
Let us note that in the computation leading to Inequality \eqref{eqIneqFKBound}, we implicitly assume that the boundary of $V$ is regular enough, so that $V^R$ is a connected open set and $\tilde v^R$ belongs to $H^1_0(V^R)$.  In order to avoid such an assumption, we can proceed as in Subsection 2.5 of \cite{cL16}: we perform the steps indicated above, in a  super-level set $V_{\alpha}=\{\tilde v^R>\alpha\}$ for some $\alpha>0$, and afterwards let $\alpha$ go to $0$. By Sard's theorem, we can find a sequence of $\alpha$'s tending to $0$ so that $V_\alpha$ is regular enough (this method was introduced in \cite{BM82}). Furthermore, as shown in Section 4 of \cite{cL16}, Inequality \eqref{eqIneqFKBound} also holds for an eigenfunction
$u$ of $-\Delta_{\Omega}^{\beta}$.

Since $\tilde u=u_1=\varphi_1^\delta\,u$, we find, after applying the Leibniz formula and Young's inequality,
\begin{equation*}
	\int_U\left|\nabla\tilde u\right|^2\,dx\le 2\,\int_{D^1_j}\left|\nabla u\right|^2\,dx+\frac{2\,C^2}{\delta^2}\int_{D^1_j}u^2\,dx,
\end{equation*}
and we have, by definition of a boundary domain,
\begin{equation*}
	\int_{U}\tilde u^2\,dx=\int_{D^1_j}u_1^2\,dx> \eps_0\,\int_{D^1_j}u^2\,dx,
\end{equation*}
for our choice of $\eps_0$.
Substituting this into Inequality \eqref{eqIneqFKBound}, we get
\begin{equation*}
	1\le \frac{4K\left|U\right|}{m_-\Lambda\,\eps_0}\left(\mu+\frac{C^2}{\delta^2}\right).
\end{equation*}
Summing over all boundary domains, we find
\begin{equation*}	\nu_1(\eps_0,u)\le\frac{4K}{m_-\Lambda\,\eps_0}\,\left|\partial\Omega_{\frac{3\delta}4}^+\right|\,\left(\mu+\frac{C^2}{\delta^2}\right).
\end{equation*}
By Inequality \eqref{eqVolTube} and since $\delta\in(0,\delta_0(\Omega)]$, we have
\begin{equation*}
	\left|\partial\Omega_{\frac{3\delta}4}^+\right|\le ML\frac{3\delta}4,
\end{equation*}
and therefore
\begin{equation}
\label{eqIneqBd}
	\nu_1(\eps_0,u)\le\frac{3KML\delta}{m_-\Lambda\,\eps_0}\,\left(\mu+\frac{C^2}{\delta^2}\right).
\end{equation}

In order to deduce from Inequalities \eqref{eqIneqBulk} and \eqref{eqIneqBd} a bound on $\nu(u)$ which is invariant by scaling, we set
\begin{equation*}
	\delta:=\left(\frac{A}{\mu}\right)^{1/4}.
\end{equation*}
We note that this choice of $\delta$ assumes that $\mu$ is large enough. That is, $\mu \geq A \delta_0(\Omega)^{-4}$.

For our choice of $\eps_0 \in (0,1)$, we therefore obtain that
\begin{multline}
\label{eqUpperNodal}
 \nu(u) \le
 \frac1\Lambda\left(\frac{1+\eps_0}{1-\eps_0}(A\mu)+\frac{3KM}{\eps_0 m_-}\rho^2(A\mu)^{3/4}+\right.\\ \left.C^2\frac{1+1/\eps_0}{1-\eps_0}(A\mu)^{1/2}+\frac{3KMC^2}{\eps_0 m_-}\rho^2(A\mu)^{1/4}\right),
\end{multline}
where the parameter $\rho$ is
the isoperimetric ratio for the domain $\Omega$ to the power $1/2$:
\begin{equation*}
\rho:=\frac{L^{1/2}}{A^{1/4}}.
\end{equation*}

\subsection{Geometry of the domain}\label{ssecgeom}

Let us first sketch how to extend the above results to a multiply connected domain. We now assume only that $\Omega$ is an open, bounded and connected set in $\RR^2$, with a $C^2$ boundary $\partial \Omega$. The set $\RR^2\setminus \overline\Omega$ has an unbounded connected component, which we denote by $D_0$. We define $b\in \NN\cup\{0\}$ as the number of bounded connected components of $\RR^2\setminus \overline\Omega$, which we denote
by $D_1,\dots,D_b$. For $h\in\{0,1,\dots,b\}$, $\partial D_h$ is a $C^2$-regular connected curve, and we denote its length by $L_h$. We set
\begin{equation*}
	s_h:=\sum_{i=0}^hL_i
\end{equation*}
for $h\in\{0,1,\dots,b\}$ and $s_{-1}:=0$. As in Section~\ref{Sec3}, we choose an arc-length parametrisation $\gamma_h:[s_{h-1},s_h]\to \RR^2$ of $\partial D_h$, oriented in such a way that for all $s\in [s_{h-1},s_h]$, the normal vector $\mathbf n_h(s)$ to $\gamma_h$ at $\gamma_h(s)$ points towards the inside of $\Omega$. Furthermore, we denote by $\kappa_h(s)$ the signed curvature of $\gamma_h$ at $\gamma_h(s)$. We keep the notation $L$ for the total length of the boundary $\partial \Omega$, so that
\begin{equation*}
	L=L_0+L_1+\dots+L_b.
\end{equation*}

We define
\begin{equation*}
	E=\bigcup_{h=0}^b(s_{h-1},s_h).
\end{equation*}
We give a natural generalisation of the definition of the function $F$ in Equation \eqref{eqF}: $F:E\times\RR\to \RR^2$ is now defined by
\begin{equation*}
	F(s,t):=\gamma_h(s)+\,t\mathbf n_h(s)
\end{equation*}
if $s\in (s_{h-1},s_h)$. Extending the definition given in Equation \eqref{eqCriticalt}, we write
\begin{equation*}
	t_+:=\left(\max_{h\in \{0,\dots,b\}}\sup_{s\in[s_{h-1},s_h]}|\kappa_h(s)|\right)^{-1},
\end{equation*}
while we continue to define $\underline\delta_+(s)$, $\underline\delta_+$ and $\delta_0(\Omega)$ by Equations \eqref{eqLocCut}, \eqref{eqCut} and \eqref{eqDelta0}. By reasoning separately on each of the connected components of $\partial\Omega$, we can show that $F$ is a $C^1$ diffeomorphism from $E\times(0,\delta_0(\Omega))$ to $\partial\Omega^+_{\delta_0(\Omega)}\setminus\bigcup_{h=0}^b\ell_h$, where the segment $\ell_h$ is defined by
\begin{equation*}
	\ell_h:=\gamma_h(s_{h-1})+(0,\delta_0(\Omega))\mathbf n_h(s_{h-1}).
\end{equation*}
Furthermore, the constructions of Section \ref{secCutOff} and the estimates of Section \ref{secEstimates} still hold.

Let us now discuss how to improve the estimates of Subsection~\ref{ssecStraightening} for domains with particular geometric properties. We have the following improvement of Corollary \ref{corVolTube}.
\begin{prop}
\label{propVolTubeGeom}
	If $\Omega$ is either simply connected or homeomorphic to an annulus, then, for all $\delta\in (0,\delta_0(\Omega)]$,
	\begin{equation*}
		\tau(\delta)\le L\delta.
	\end{equation*}
\end{prop}

\begin{proof}	Let us first consider the simply connected case. From the change of variable $x=F(q)$, we obtain
\begin{equation*}
\tau(\delta)=\int_0^L\int_0^\delta\left(1-t\,\kappa(s)\right) \, dt \, ds
=L\,\delta-\frac{\delta^2}2\int_0^L\kappa(s)\,ds.
\end{equation*}
Since $\gamma$ is a simple, closed, positively-oriented curve,
\begin{equation*}
	\int_0^L\kappa(s)\,ds=2\,\pi,
\end{equation*}
giving the desired inequality (see for instance Corollary 9.5.2 of \cite{BG88}).
In the case where $\Omega$ is homeomorphic to an annulus, $\partial \Omega$ has two connected components. After a similar computation for each of them, we find
\begin{equation*}
\tau(\delta)=L\,\delta-\frac{\delta^2}2\left(\int_0^{s_0}\kappa_0(s)\,ds+\int_{s_0}^{s_1}\kappa_1(s)\,ds\right).
\end{equation*}
Since the normal vector has been chosen to point inwards, $\gamma_0$ and $\gamma_1$ have opposite orientations and the second term in the right-hand side vanishes. The desired inequality is an equality in this case.
\end{proof}

\begin{rem}\label{remSzNagy}
Suppose $\Omega$ is a multiply-connected domain that is homeomorphic to a disc with
$b \in \NN\cup \{0\}$ discs removed. By the Gauss-Bonnet theorem, we then have
\begin{equation*}
\int_{0}^{L} \kappa(s) \, ds = \int_{\partial \Omega} k_{g}(s) \, ds
=2\pi \chi(\Omega) = 2\pi(1-b),
\end{equation*}
where $\chi(\Omega)$ is the Euler Characteristic of $\Omega$ and $k_{g}$ denotes the geodesic curvature.
If the width of the interior tubular neighbourhood is sufficiently small, that is $\delta \le
\delta_0(\Omega)$, then we obtain
\begin{equation*}
\tau(\delta)= L\delta - \pi(1-b)\delta^2.
\end{equation*}
We recover the cases $b=0,1$ as discussed above.
If $b \geq 2$, then we have that
\begin{equation*}
\tau(\delta)= L\delta + \pi(b-1)\delta^2> L\delta,
\end{equation*}
for all $\delta \in(0,\delta_0(\Omega)]$. The quantity $L\delta$ is no longer an upper bound of $\tau(\delta)$, but is the leading order term for $\delta$ small.

A similar upper bound was obtained in \cite{SzN}, but there the width of the interior tubular neighbourhood
need only be smaller than the in-radius of the domain (see also \cite{CB83}).
\end{rem}

\subsection{Application to convex domains}\label{ss:5.4}
In this subsection, we assume that $\Omega$ is a convex domain.
We show that Lemma~\ref{lemEstInt} and Proposition~\ref{propIneqQuot} hold for convex domains, and give explicit constants in the corresponding bounds. We first make the following remark.

\begin{rem}\label{remconvex}
Let us assume that $\Omega$ is convex. Then $\delta_0(\Omega)=t_+(\Omega)$. Indeed, this is equivalent to saying that a ball of radius smaller than $t_+(\Omega)$ can roll freely inside the set $\Omega$.  This follows from a more general result that was established in \cite{BS89}, where this particular case is discussed on page 53, in answer to a question of J.A. Delgado \cite{Del}.  Furthermore, Inequality \eqref{eqVolTube} holds with $M=1$ for any $r>0$. Indeed, let us define the inner parallel set at distance $r>0$ by $\Omega_r:=\Omega\setminus \partial \Omega^+_r$. The set $\Omega_r$ is convex. We write $m(r):=|\Omega_r|$ and we have, by definition, $\tau(r):=|\Omega|-m(r)$ for all $r>0$. According to \cite{Mat}, the function $r\to m(r)$ is differentiable and $m'(r)=-|\partial \Omega_r|$. Therefore, for $r>0$,
\begin{equation*}
	\tau(r)=\int_0^r|\partial \Omega_t|\,dt.
\end{equation*}
The desired inequality then follows from the fact that the perimeter is non-decreasing with respect to inclusion among convex sets.  We note that both results in the present remark hold in arbitrary dimension.
\end{rem}

\begin{prop}\label{prop5.8}
If $\Omega \subset \RR^2$ is convex, then
\begin{enumerate}
  \item Lemma~\ref{lemEstInt} holds with $m_-=1/4$ and $m_+=1$. That is, if $V$ is an open set
$V\subset (0,L)\times(0,3\,t_{+}(\Omega)/4)$ and $U=F(V)$ with $F$ as defined in \eqref{eqF}, then for any measurable and non-negative function $g:V\to \RR$,
\begin{equation*}
	\frac14 \int_Vg\,dq\le \int_U f\,dx \le \int_Vg\,dq,
\end{equation*}
where $f=g\circ F^{-1}$.
  \item Proposition \ref{propIneqQuot} holds with $K=4$. That is, if $V$ is an open set
$V\subset (0,L)\times(0,3\,t_{+}(\Omega)/4)$, $U=F(V)$, $u\in H^1(U)$ and $v:=u\circ F$, we have
\begin{equation*}
	\frac{\int_V\left|\nabla v\right|^2\,dq}{\int_Vv^2\,dq}\le 4\,\frac{\int_U\left|\nabla u\right|^2\,dx}{\int_Uu^2\,dx}.
\end{equation*}
\end{enumerate}
\end{prop}

\begin{proof} Since $\Omega$ is convex, we have $\kappa(s)\ge0$ for all $s\in[0,L]$, so that
\begin{equation*}
	1/4\le1-t\,\kappa(s)\le 1
\end{equation*}
for all $t\in (0,3 t_{+}(\Omega)/4)$. Using these inequalities in the change of variable formula, we obtain the desired results.
\end{proof}

\section{An upper bound for $R_\Omega^D(\mu)$}
\label{secRemainder}

As discussed in Section~\ref{Sec2}, in order to obtain an upper bound for the Courant-sharp Neumann (or Robin)
eigenvalues, we require an upper bound for the remainder of the Dirichlet counting function $R_{\Omega}^D(\mu)$.

In \cite{vdBkG16cs}, the authors obtain an upper bound for $R_\Omega^D(\mu)$ by using an inner partition of $\Omega$ into squares of uniform side-length and Dirichlet bracketing. Suppose these squares have side-length $\ell$. Then, in the notation of the present article, Inequality (13) of \cite{vdBkG16cs} reads:
\begin{equation*}
R_{\Omega}^D(\mu) \leq \frac{\tau(\sqrt{2}\ell)}{4\pi}\mu + \frac{|\Omega|}{\sqrt{2}\ell} \mu^{1/2}.
\end{equation*}
Corollary~\ref{corVolTube} gives an upper bound for the volume of the inner tubular neighbourhood provided
that the width is sufficiently small.
Choosing $\ell>0$ so that $\sqrt2\ell\le 3\delta_0(\Omega)/4$, and applying Corollary~\ref{corVolTube}, we obtain
\begin{equation*}
R_{\Omega}^D(\mu) \le \frac{\sqrt2M\ell L}{4\pi |\Omega|}(|\Omega|\mu)+\frac{|\Omega|^{1/2}}{\sqrt2\ell}(|\Omega|\mu)^{1/2}.
\end{equation*}

We then use the freedom in the choice of $\ell$ to minimise the right-hand side, meaning that we set
\begin{equation*}
\ell:=\sqrt\frac{2\pi |\Omega|}{ML}\mu^{-1/4}
\end{equation*}
to get
\begin{equation}
\label{eqRemainder}
R_\Omega^D(\mu)\le \sqrt\frac{ML|\Omega|}{\pi}\mu^{3/4}=\sqrt{\frac{M}\pi}\rho(\Omega)(|\Omega|\mu)^{3/4},
\end{equation}
as soon as $\mu$ is large enough for $\sqrt2\ell\le3\delta_0(\Omega)/4$ to be satisfied.

Since we used Corollary~\ref{corVolTube} in this section, Inequality \eqref{eqRemainder} holds with $M=7/4$, for all $\Omega\subset\RR^2$ with a $C^2$ boundary, as soon as $\mu$ is large enough for the condition $\sqrt2\ell\le3\delta_0(\Omega)/4$ to be satisfied. If we additionally assume $\Omega$ to be simply connected or homeomorphic to an annulus, Proposition \ref{propVolTubeGeom} tells us that Inequality \eqref{eqRemainder} holds with $M=1$ as soon as $\mu$ is large enough so that $\sqrt2\ell\le\delta_0(\Omega)$. Finally, if we assume $\Omega$ to be convex,  Inequality \eqref{eqRemainder} holds with $M=1$ for all $\mu>0$.

\section{Upper bound for Courant-sharp eigenvalues}\label{Sec7}

We again set $A:=|\Omega|$ to ease the notation. Let us consider an eigenpair $(\mu,u)$ such that $u$ is a Courant-sharp eigenfunction.
We assume that $\mu$ is large enough that
\begin{equation}
\label{eqIneqLength}
	\delta=\left(\frac{A}\mu\right)^{1/4}\le \delta_0(\Omega)
\end{equation}
and
\begin{equation}
\label{eqIneqSide}
	\sqrt2\ell=\sqrt\frac{4\pi A}{ML}\mu^{-1/4}\le\frac34 \delta_0(\Omega).
\end{equation}

We recall that $\eps_0 \in (0,1)$ is fixed as in Subsection~\ref{ssecNoofnodaldomains}.
By substituting Inequality \eqref{eqUpperNodal}, Equation \eqref{eqDefRemainder} and Inequality \eqref{eqRemainder} into Inequality \eqref{eqBasicNC}, we obtain the necessary condition
\begin{multline}
\label{eqNC}
\left(\frac1{4\pi}-\frac{1+\eps_0}{(1-\eps_0)\Lambda}\right)(A\mu)-\left(\frac{3K M}{\eps_0 m_-\Lambda}\rho^2+\sqrt{\frac{M}\pi}\rho\right)(A\mu)^{3/4}\\ -C^2\frac{1+1/\eps_0}{(1-\eps_0)\Lambda}(A\mu)^{1/2}-\frac{3KMC^2}{\eps_0 m_-\Lambda}\rho^2(A\mu)^{1/4}<0.
\end{multline}
We can reformulate Inequality \eqref{eqNC} by saying that if $\mu$ is Courant-sharp and satisfies \eqref{eqIneqLength} and \eqref{eqIneqSide},
\begin{equation*}
f_\rho(\xi):=a_0\xi^3+a_1(\rho)\xi^2+a_2\xi+a_3(\rho)<0,
\end{equation*}
with
\begin{align*}
a_0&:=\left(\frac1{4\pi}-\frac{1+\eps_0}{(1-\eps_0)\Lambda}\right);\\
a_1(\rho)&:=-\left(\frac{3K M}{\eps_0 m_-\Lambda}\rho^2+\sqrt{\frac{M}\pi}\rho\right);\\
a_2&:=-C^2\frac{1+1/\eps_0}{(1-\eps_0)\Lambda};\\
a_3(\rho)&:=-\frac{3KMC^2}{\eps_0 m_-\Lambda}\rho^2;\\
\xi&:=(A\mu)^{1/4}.
\end{align*}
Therefore, we have that
\begin{equation*}
A\mu\le \xi^*(\rho)^4,
\end{equation*}
where $\xi^*(\rho)$ is the largest real zero of the function $\xi\mapsto f_\rho(\xi)$. We summarise the above discussion in the following proposition.

\begin{prop}\label{propmubound}
If $\mu$ is a Courant-sharp eigenvalue of the Neumann Laplacian on $\Omega$ or of $-\Delta_{\Omega}^{\beta}$,
we have
\begin{equation*}
\mu\le\max\left(\frac{4096\pi^2|\Omega|}{81M^2\rho(\Omega)^4\delta_0(\Omega)^4},\frac{|\Omega|}{\delta_0(\Omega)^4},\frac{\xi^*(\rho(\Omega))^4}{|\Omega|}\right).
\end{equation*}
\end{prop}

We note that the function $f_\rho$, and therefore also $\xi^*(\rho)$, only depend on the geometric properties of the domain $\Omega$ and the choice of $\eps_0$. Sections \ref{secCutOff} and \ref{secEstimates} allow us to specify several sets of possible choices of constants in the definition of $f_\rho$, according to the nature of the domain $\Omega$, which we always assume to be open, bounded and connected with a $C^2$ boundary. We summarise them in Table \ref{tabConst}, where $j_{0,1}$ denotes the smallest positive zero of the Bessel function of the first kind $J_0$.

\begin{table}
	\centering
	\caption{Possible choices for the constants\label{tabConst}}
	\begin{tabular}{|c|c|c|c|c|c|c|}
	\hline
	Nature of $\Omega$&$\Lambda$&$C$&$m_-$&$\eps_0$&$M$&$K$\\
	\hline
    General case&$\pi j_{0,1}^2$&$2\sqrt3$&$1/4$&$1/6$&$7/4$&$7$\\
       Simply or doubly connected&$\pi j_{0,1}^2$&$ 2\sqrt3$&$1/4$&$1/6$&$1$&$7$
       \\
    Convex&$\pi j_{0,1}^2$&$2\sqrt3$&$1/4$&$1/6$&$1$&$4$\\
    \hline
    \end{tabular}
\end{table}

\section{Geometric upper bounds for Courant-sharp eigenvalues of convex planar domains}\label{Sec8}
\subsection{Geometric upper bounds}\label{ss:8.1}

In this subsection, we suppose that $\Omega$ is an open, bounded, convex planar domain with $C^2$ boundary. To simplify the notation, we set $A:=\vert \Omega \vert$, $t_+:=t_+(\Omega)$ and $\rho := \rho(\Omega)$. Let $\mu$ be a Courant-sharp eigenvalue of the Neumann Laplacian on $\Omega$, or of $-\Delta_{\Omega}^{\beta}$. We assume that $\mu$ is large enough so that Inequality \eqref{eqIneqLength} is satisfied. We recall that $\eps_0 \in (0,1)$ is fixed as in Subsection~\ref{ssecNoofnodaldomains}.

In order to obtain upper bounds for a Courant-sharp eigenvalue $\mu$ in terms of some of the geometric
quantities of the domain $\Omega$, we rewrite Inequality \eqref{eqNC}, with the constants from the last row of Table~\ref{tabConst}, as follows.

\begin{multline*}
\left(\frac1{4\pi}-\frac{1+\eps_0}{(1-\eps_0)\Lambda}\right)(A\mu)\\
\leq \left(\frac{3K M}{\eps_0 m_-\Lambda}\rho^2+\sqrt{\frac{M}\pi}\rho +\frac{C^2}{( A\mu )^{1/4}}\frac{1+1/\eps_0}{(1-\eps_0)\Lambda}+\frac{3KMC^2}{\eps_0 m_-\Lambda}\frac{\rho^2}{(A\mu)^{1/2}}
\right)(A\mu)^{3/4},
\end{multline*}
which implies that
\begin{multline}
\label{eqNCmu1}
A\mu \leq \left(\frac1{4\pi}-\frac{1+\eps_0}{(1-\eps_0)\Lambda}\right)^{-4} \\
\times \left(\frac{3K M}{\eps_0 m_-\Lambda}\rho^2+\sqrt{\frac{M}\pi}\rho +\frac{C^2}{(A\mu)^{1/4}}\frac{1+1/\eps_0}{(1-\eps_0)\Lambda}
+\frac{3KMC^2}{\eps_0 m_-\Lambda}\frac{\rho^2}{(A\mu)^{1/2}}\right)^4.
\end{multline}

As $\Omega$ is convex, $\delta_0(\Omega) = t_{+}$ (see Remark~\ref{remconvex}). In particular, there exists a ball
of radius $t_+$ which is contained in $\Omega$, that we denote by $B_{t_{+}}$.
Then $A= \vert \Omega \vert \geq \vert B_{t_{+}} \vert = \pi t_{+}^2$, so $A t_{+}^{-2} \geq \pi$.
Together with Inequality \eqref{eqIneqLength}, we then have
\begin{equation*}
 A\mu \geq \left(\frac{A}{t_{+}^2}\right)^2 \geq \pi^2.
\end{equation*}
Substituting this into Inequality \eqref{eqNCmu1}, we obtain
\begin{equation}
\label{eqboundt*}
A\mu \leq D_2^{-4}
\left(\frac{288 \rho^2}{\Lambda} + \frac{\rho}{\pi^{1/2}} + \frac{504}{5 \pi^{1/2} \Lambda}
+ \frac{3456 \rho^2}{\pi \Lambda}\right)^4,
\end{equation}
which implies
\begin{equation}\label{eqdefL1}
\mu \leq D_2^{-4}A^{-1}
\left(\frac{288 \rho^2}{\Lambda} + \frac{\rho}{\pi^{1/2}} + \frac{504}{5\pi^{1/2} \Lambda}
+ \frac{3456 \rho^2}{\pi \Lambda}\right)^4 =: L_1(\rho, A),
\end{equation}
where $D_2= \frac1{4\pi}-\frac{7}{5\Lambda}$.

With Proposition~\ref{propmubound} and the preceding discussion in mind, we have that, if $\mu$
is Courant-sharp, then
\begin{equation}
\label{eqUBmu2D}	
\mu\le\max\left(\frac{A}{t_{+}^4}, L_1(\rho, A)\right).
\end{equation}
Note that since $\Omega$ is convex, Inequality \eqref{eqVolTube} holds for all $r>0$ (see Remark~\ref{remconvex}), so Inequality \eqref{eqIneqSide} is not necessary in the convex case.\\

In what follows, we obtain an upper bound for $\mu$ that involves some of the other geometric quantities
of $\Omega$. We consider Inequality \eqref{eqNCmu1} and use the fact that $\mu=\mu_k(\Omega) \geq \mu_2(\Omega)$
for $k \geq 2$. We can then make use of known geometric estimates for the first positive Neumann eigenvalue of $\Omega$.
This is sufficient for our purposes due to the monotonicity of the Robin eigenvalues with respect to the function $\beta$. That is, for $k \in \NN$, $\mu_k(\Omega) \leq \mu_k(\Omega,\beta)$.

From Inequality \eqref{eqNCmu1}, we have
\begin{equation}
\label{eqNCAmu}
A\mu \leq D_2^{-4}
\left(\frac{288 \rho^2}{\Lambda} + \frac{\rho}{\pi^{1/2}} + \frac{504}{5\Lambda (A\mu_2)^{1/4}}
+ \frac{3456 \rho^2}{\Lambda (A\mu_2)^{1/2}}\right)^4,
\end{equation}
which implies that
\begin{equation}
\label{eqNCmu}
\mu \leq D_2^{-4}
\left(\frac{288 \rho^2}{\Lambda A^{1/4}} + \frac{\rho}{\pi^{1/2}A^{1/4}} + \frac{504}{5\Lambda A^{1/2}\mu_2^{1/4}}
+ \frac{3456 \rho^2}{\Lambda A^{3/4}\mu_2^{1/2}}\right)^4.
\end{equation}

Since $\Omega$ is convex, we invoke the $2$-dimensional version of the classical inequality due to Payne and Weinberger, \cite{PW}, for the first positive Neumann eigenvalue of $\Omega$,
\begin{equation}
\label{eqPW}
\mu_2(\Omega) \geq \frac{\pi^2}{\diam(\Omega)^2},
\end{equation}
where $\diam(\Omega)$ denotes the diameter of $\Omega$, to obtain
\begin{align}
\label{eqmuPW}
\mu &\leq D_2^{-4}\left(\frac{288 \rho^2}{\Lambda A^{1/4}} + \frac{\rho}{\pi^{1/2}A^{1/4}} + \frac{504\diam(\Omega)^{1/2}}{5\Lambda \pi^{1/2}A^{1/2}} + \frac{3456 \rho^2 \diam(\Omega)}{\pi \Lambda A^{3/4}}\right)^4 \notag\\
& \ \ \ =: L_2(\rho, A, \diam(\Omega)).
\end{align}

By Proposition~\ref{propmubound} and the fact that $\Omega$ is convex, we have the following inequality:
\begin{equation*}	
\mu\le\max\left(\frac{A}{t_{+}^4},L_2(\rho, A, \diam(\Omega))\right).
\end{equation*}

\begin{rem}
 It is  also  possible to obtain an upper bound for the largest Courant-sharp eigenvalue of $\Omega$ involving only the geometric quantities $A$, $\rho$ and $t_{+}$ via the above approach by making use of Lemma 5.2 of \cite{M09} with $K=\Omega$ and $L$ a ball of radius $t_{+}$ that is contained in $\Omega$.
\end{rem}

We collect the preceding upper bounds in the following proposition.
\begin{prop} \label{prop:Nevalconvex}
Let $\Omega$ be an open, bounded and convex set with a $C^2$  boundary. If $\mu$ is a Courant-sharp eigenvalue of the Neumann Laplacian on $\Omega$ or of $-\Delta_{\Omega}^{\beta}$, then we have that
\begin{equation*}
\mu\le\max\left(\frac{\vert \Omega \vert}{t_{+}(\Omega)^4}, L_1(\rho(\Omega), \vert \Omega \vert)\right),
\end{equation*}
and
\begin{equation*}	
\mu\le\max\left(\frac{\vert \Omega \vert}{t_{+}(\Omega)^4}, L_2(\rho(\Omega), \vert \Omega \vert, \diam(\Omega))\right),
\end{equation*}
where the functions $L_1$ and $L_2$ are defined in \eqref{eqdefL1} and \eqref{eqmuPW} respectively.
\end{prop}

Suppose that $\mu$ is large enough so that \eqref{eqIneqLength} holds.
Note that since $\Omega$ is convex, $\diam(\Omega) \leq \frac{L}{2}$ and so $\mu_2(\Omega) \geq \frac{4\pi^2}{L^2}$, by the Payne-Weinberger inequality.
Substituting this into \eqref{eqNCAmu}, we obtain
\begin{equation*}
\mu \leq \frac{\rho^4}{A D_2^4}
\left(\frac{288 \rho}{\Lambda} + \frac{1}{\pi^{1/2}} + \frac{504}{5(2\pi)^{1/2}\Lambda}
+ \frac{3456 \rho^3}{2\pi \Lambda}\right)^4
=: A^{-1} D_2^{-4}\rho^4 G(\rho)^4,
\end{equation*}
where $\rho \mapsto G(\rho)$ is an increasing function for $\rho \geq 0$.

In fact, in this convex case, we can deduce a simplified upper bound for $\mu$ and obtain a geometric
upper bound for the number of Courant-sharp eigenvalues.

\begin{prop} \label{prop:Noevalconvex}
Let $\Omega$ be an open, bounded and convex set with a $C^2$  boundary.
Let $\mu=\mu_k(\Omega)$ or $\mu=\mu_k(\Omega,\beta)$ where $\beta$ is a non-negative Lipschitz continuous function on $\partial \Omega$. There exist positive constants $C$ and $C'$ that do not depend on $\Omega$ such that,
if $\mu$ is Courant-sharp, then
\begin{equation*}
 \mu\le C\left(\frac{\vert \Omega \vert}{t_+(\Omega)^4}+\frac{\rho(\Omega)^8}{\vert \Omega \vert}\right),
\end{equation*}
and
\begin{equation*}
	k\le C'\left(\frac{\vert \Omega \vert^2}{t_+(\Omega)^4}+\rho(\Omega)^8\right).
\end{equation*}
\end{prop}

The first inequality of Proposition~\ref{prop:Noevalconvex} follows from Inequalities \eqref{eqboundt*} and \eqref{eqUBmu2D}, by using the Isoperimetric inequality $\rho\ge\sqrt2\pi^\frac14$. The second inequality follows from the first one and Corollary~\ref{corUBk}, together with Remark~\ref{remAppRobin}, after applying Young's inequality repeatedly.
Indeed, by following this procedure, we obtain
\begin{equation*}
C = D_2^{-4} \left(\frac{288}{\Lambda} + \frac{1}{\sqrt{2}\pi^{3/4}} + \frac{35064}{10 \pi \Lambda}\right)^4 \approx 8.98 \times 10^{17},
\end{equation*}
and $C' \leq C+1$.

We remark that this proves Proposition~\ref{prop:1.2} since the latter is Proposition~\ref{prop:Noevalconvex} restricted to the Neumann case.

\subsection{Application to the disc of unit area}

For the disc $D \subset \RR^2$ of unit area, we have $\rho(D)=(4\pi)^{1/4}$ and $\delta_0(D)= t_{+}(D)=\pi^{-1/2}$.
We wish to evaluate Proposition~\ref{propmubound}, so we must first calculate $\xi^*(\rho)$.
We use the last line of Table \ref{tabConst}.
The value of $\xi^*(\rho)$ can be computed either numerically or with Cardano's formula.
We find, for any Courant-sharp Neumann (or Robin) eigenvalue $\mu(D)$ of $D$,
\begin{equation*}
\mu(D) \le 2.67\times10^{17}.
\end{equation*}

It is well-known that the first positive Neumann eigenvalue of $D$ is $\pi (j_{1,1}')^2$ where $j_{1,1}'$ is the first positive zero of $J_1'$, the derivative of the Bessel function $J_1$. By substituting this into Inequality \eqref{eqNCmu}, we obtain that any Courant-sharp Neumann eigenvalue $\mu(D)$ of $D$ satisfies
\begin{equation*}
 	\mu(D) \leq 1.26\times 10^{20}.
\end{equation*}

We also have $\diam(D) = 2\pi^{-1/2}$. By Inequality \eqref{eqmuPW}, we have that any Courant-sharp Neumann (or Robin) eigenvalue $\mu(D)$ of $D$ satisfies
\begin{equation*}
	\mu(D) \leq 2.09 \times 10^{20}.
\end{equation*}

By comparing the two preceding upper bounds, we remark that we do not lose a great deal by appealing to the bound of Payne and Weinberger, even though it is not the natural choice for our situation. Indeed, equality is achieved in \eqref{eqPW} for a rectangle of fixed area with one side shrinking to $0$. Such shrinking behaviour is not possible in our situation since $\Omega$ contains a ball of radius $t_{+}(\Omega)>0$.

We now evaluate the first inequality of Proposition~\ref{prop:Nevalconvex} for the disc of unit area.
We obtain that any Courant-sharp Neumann (or Robin) eigenvalue $\mu(D)$ of $D$ satisfies
\begin{equation*}
	\mu(D) \leq 1.42 \times 10^{20},
\end{equation*}
which is better than the second inequality of Proposition~\ref{prop:Nevalconvex} for the disc of unit area,
but worse than using Inequality \eqref{eqNCmu} directly (see above).

In addition, Proposition~\ref{prop:Noevalconvex} applied to the disc of unit area gives that
any Courant-sharp Neumann (or Robin) eigenvalue $\mu(D)$ of $D$ satisfies
\begin{equation*}
	\mu(D) \leq 1.51 \times 10^{20},
\end{equation*}
and the number of such eigenvalues $k \leq 1.51 \times 10^{20}$.

For comparison with the above upper bounds, we recall that the third positive Neumann eigenvalue is the largest Courant-sharp Neumann eigenvalue as proved in \cite{HPS2}. In addition, $\mu_4(D) = \pi (j_{2,1}')^2 \approx 3.054^2 \pi < 30$ (where $j_{2,1}'$ is the first positive zero of $J_2'$, the derivative of the Bessel function $J_2$).
The Courant-sharp Robin eigenvalues of $D$ are not known (to date).

\section{Generalisation to arbitrary dimension}\label{SecNdim}

\subsection{Tubular neighbourhoods of the boundary}\label{SubsecTubesNdim}

We now assume that $\Omega$ is a bounded, open and connected set in $\RR^n$, with $n\ge 3$, such that $\partial \Omega$ is a $C^2$ submanifold of $\RR^n$. We extend the methods used in the previous sections to this situation.

To simplify the notation, we write $\Gamma$ for the submanifold $\partial \Omega$. As before, for $r>0$, we define the inner tubular neighbourhood with radius $r$:
\begin{equation*}
	\Gamma^+_r:=\{x\in\Omega\,;\,\mbox{dist}(x,\Gamma)<r\}.
\end{equation*}
 For all $x'\in \Gamma$, we denote by $\mathbf n(x')$ the outward unit normal vector at $x'$. In order to parametrise the inner tubular neighbourhoods, we introduce the normal bundle of $\Gamma$. Furthermore, for all $x'\in\Gamma$, we identify the one-dimensional vector space spanned by $\mathbf n(x')$ with the real line $\RR$. We therefore write the normal bundle as a trivial product:
\begin{equation*}
	N(\Gamma):=\Gamma\times\RR
\end{equation*}
and we define a $C^1$ mapping $F: N(\Gamma)\to \RR^n$ by
\begin{equation}
\label{eqFDimN}
	F(x',t)=x'-t\mathbf n(x').
\end{equation}

Let us recall that the mapping $x'\mapsto \mathbf n(x')$ from $\Gamma$ to $\mathbb S^{n-1}$ is known as the \emph{Gauss map}  (see for instance Subsection 2.5 of \cite{Sch}). Its differential is a symmetric linear endomorphism of $T_{x'}\Gamma$, the tangent space to $\Gamma$ at $x'$, seen as a  subspace of $\RR^n$. It is called the \emph{Weingarten endomorphism}, and we denote it by $W_{x'}$. Its eigenvalues are called the \emph{principal curvatures} of $\Gamma$ at $x'$. Let $(\mathbf e_1,\dots,\mathbf e_{n-1})$ be an orthonormal basis of $T_{x'}\Gamma$ consisting of eigenvectors of $W_{x'}$ and let $(\kappa_1(x'),\dots,\kappa_{n-1}(x'))$ be the associated principal curvatures. Then $\mathcal (\mathbf e_1,\dots,\mathbf e_{n-1},1)$ and $(\mathbf e_1,\dots,\mathbf e_{n-1},\mathbf n(x'))$ are orthonormal bases of $T_{(x',t)}N(\Gamma)$, identified with $T_{x'}\Gamma\oplus \RR$, and $\RR^n$ respectively, and the matrix of the differential $D_{(x',t)}F$ in those bases is
\begin{equation*}
\left(
\begin{array}{cccc}
1-t\kappa_1(x')&\cdots&0&0\\
\vdots&\ddots&\vdots&\vdots\\
0&\cdots&1-t\kappa_{n-1}(x')&0\\
0&\cdots&0& -1
\end{array}
\right).
\end{equation*}

We now define
\begin{equation*}
	t_+(\Omega):=\left(\sup_{x'\in\Gamma}\,\max_{i\in\{1,\dots,n-1\}}\left|\kappa_i(x')\right|\right)^{-1},
\end{equation*}
and, for any $x'\in\Gamma$, the (internal) \emph{cut-distance} to $\Gamma$ at $x'$:
\begin{equation*}
	\delta_+(x'):=\sup\{\delta>0\,;\,\mbox{dist}(F(x',t),\Gamma)=t\mbox{ for all }t\in[0,\delta]\}.
\end{equation*}
We set
\begin{equation*}
 \underline\delta_+(\Omega):=\inf_{x'\in\Gamma}\delta_+(x')
\end{equation*}
and
\begin{equation*}
	\delta_0(\Omega):=\min\{t_+(\Omega),\,\underline\delta_+(\Omega)\}.
\end{equation*}
By construction, $F$ is a diffeomorphism of class $C^1$ from $\Gamma\times(0,\delta_0(\Omega))$ to $\Gamma^+_{\delta_0(\Omega)}$.

\begin{lem}\label{lemIntegrals}
Let $V$ be an open set
in $\Gamma\times(0,\delta_0(\Omega))$ and let $U=F(V)$. For any measurable, non-negative function $g:V\to \RR$,
\begin{equation*}
	\int_U f(x)\,dx=\int_{\Gamma}\left(\int_0^{+\infty}\mathbf 1_{V}(x',t) g(x',t)\prod_{i=1}^{n-1}\left(1-t\kappa_i(x')\right)\,dt\right)\,dx'
\end{equation*}
where $f=g\circ F^{-1}$, $\mathbf 1_{V}$ is the characteristic function of $V$ and $dx'$ is the surface measure in $\Gamma$ induced by the Lebesgue measure in $\RR^n$.
\end{lem}

\begin{proof} The proof is similar to the proof of Weyl's formula for the volume of tubes given in Chapter 6 of \cite{BG88}. First, we use Proposition 3.3.16 of \cite{BG88} to get
\begin{equation*}
	\int_Uf(x)\,dx=\int_Vg\,F^*(dx),
\end{equation*}
where $dx$ denotes the Lebesgue measure in $\RR^n$ and $F^*(dx)$ the pullback of $dx$ by the diffeomorphism $F$.

Then, we note that, up to a change of sign, the mapping $F$ defined in Equation \eqref{eqFDimN} is the canonical map defined in Equation  2.7.5 of \cite{BG88}. According to Equation 6.8.2 of \cite{BG88}, there exists a function $G$ defined on $\Gamma\times(0,\delta_0(\Omega))$, to be specified below, such that
\begin{equation*}
	F^*(dx)=G(x',t)\,dx'\otimes dt,
\end{equation*}
where $dx'\otimes dt$ denotes the product measure in $N(\Gamma)=\Gamma\times\RR$, with $dt$ the Lebesgue measure in $\RR$. From Equation 6.7.16 of \cite{BG88}, we obtain
\begin{multline*}
	\int_{V}g\,F^*(dx)=\int_{\Gamma\times\RR}\mathbf 1_V(x',t)g(x',t)G(x',t)\,dx'\otimes dt=\\ \int_{\Gamma}\left(\int_{\RR}\mathbf 1_V(x',t)g(x',t)G(x',t)\,dt\right)\,dx'.
\end{multline*}

It remains to give an explicit formula for $G(x',t)$. We fix $u\mapsto \mathbf x(u)$ a local parametrisation of $\Gamma$ such that $\mathbf x(0)=x'$ and, for $i\in\{1,\dots,n-1\}$,
\begin{equation*}
	\frac{\partial \mathbf x}{\partial u_i}(0)=\mathbf e_i.
\end{equation*}
From Equation 6.8.9 and Proposition 6.6.2 of \cite{BG88}, we obtain
\begin{equation*}
G(x',t)=\frac{\det\left(\Pi\left(\frac{\partial \mathbf x}{\partial u_1}(0)-t\frac{\partial\mathbf n\circ \mathbf x}{\partial u_1}(0)\right),\dots,\Pi\left(\frac{\partial \mathbf x}{\partial u_{n-1}}(0)-t\frac{\partial\mathbf n\circ \mathbf x}{\partial u_{n-1}}(0)\right)\right)}{\det\left(\Pi\left(\frac{\partial \mathbf x}{\partial u_1}(0)\right),\dots,\Pi\left(\frac{\partial \mathbf x}{\partial u_{n-1}}(0)\right)\right)},
\end{equation*}
where $\Pi$ is the orthogonal projection from $\RR^n$ to $T_{x'}\Gamma$. It follows from the chain rule that, for all $i\in\{1,\dots,n-1\}$,
\begin{equation*}
	\frac{\partial\mathbf n\circ \mathbf x}{\partial u_i}(0)=W_{x'}[\mathbf e_i]=\kappa_i(x') \mathbf e_i.
\end{equation*}
This finally implies
\begin{equation*}
	G(x',t)=\prod_{i=1}^{n-1}(1-t\kappa_i(x')).\qedhere
\end{equation*}
\end{proof}

Lemma \ref{lemIntegrals} allows us to extend the estimates of Section \ref{secEstimates} to higher dimensions in a straightforward manner. Indeed, Lemma \ref{lemEstInt}, Corollaries \ref{corVol} and \ref{corVolTube}, and Proposition \ref{propIneqQuot} hold with the constants given in Table \ref{tabConstndim}.
\begin{table}
	\centering
	\caption{Possible choices for the constants in $\RR^n$ \label{tabConstndim}}
	\begin{tabular}{|c|c|c|c|c|}
	\hline
	Nature of $\Omega$&$C$&$m_-$&$M$&$K$\\
	\hline
    General case&$2\sqrt3$&$1/4^{n-1}$&$(7/4)^{n-1}$&$7^{n-1}$\\
    Convex&$2\sqrt3$&$1/4^{n-1}$&$1$&$4^{n-1}$\\
    \hline
    \end{tabular}
\end{table}

\subsection{Upper bounds for Courant-sharp eigenvalues}

To simplify the formulas, we write $V:=|\Omega|$ and $S:=|\partial \Omega|$ (the total $(n-1)$-dimensional surface measure of $\partial \Omega$) in the rest of this section. We denote by $N_\Omega$ either $N_\Omega^N$ or $N_\Omega^\beta$, since the lower bound that we derive holds for both types of counting function.

Following \cite{HPS2}, we define
\begin{equation*}
	\gamma(n):=\frac{(2\pi)^n}{\omega_n\Lambda(n)^\frac{n}2},
\end{equation*}
and we recall that $\gamma(n)<1$ for all integers $n\ge2$.
For the rest of this section, we set
\begin{equation}
\label{eqEpsNDim}
	\eps_0:=\frac12\cdot\frac{1-\gamma(n)^\frac2n}{1+\gamma(n)^\frac2n}.
\end{equation}
This choice of $\eps_0$ will be explained below.

We consider $(\mu,u)$, an eigenpair of the Neumann Laplacian on $\Omega$ or of $-\Delta_\Omega^\beta$, with $\mu$ large enough, in a sense to be made precise.

Following the steps of Subsection \ref{ssecNoofnodaldomains}, we obtain an upper bound  for the number of bulk domains
\begin{equation*}
	\nu_0(\eps_0,u)\le \frac{V}{\Lambda(n)^{\frac{n}2}}\left(\frac{1+\eps_0}{1-\eps_0}\mu+\left(\frac{1+\frac{1}{\eps_0}}{1-\eps_0}\right)\frac{C^2}{\delta^2}\right)^{\frac{n}2},
\end{equation*}
where $\Lambda(n)$ is the first eigenvalue of the Dirichlet Laplacian on a ball of volume $1$ in $\RR^n$.
More explicitly,
\begin{equation*}
\Lambda(n):=\omega_n^\frac2n j_{\frac{n}2-1,1}^2,
\end{equation*}
where $\omega_n$ denotes the volume of a ball of radius $1$ in $\RR^n$ and $j_{\frac{n}2-1,1}$ the smallest positive zero of the Bessel function of the first kind $J_{\frac{n}2-1}$.
We also obtain the following upper bound for the number of boundary domains
\begin{equation*}	
	\nu_1(\eps_0,u)\le \frac{3\cdot2^{\frac{n}2-1}K^{\frac{n}2}SM\delta}{m_-\Lambda(n)^{\frac{n}2}\eps_0^{\frac{n}2}}\left(\mu+\frac{C^2}{\delta^2}\right)^{\frac{n}2}.
\end{equation*}
Both bounds hold assuming $ \delta\in(0,\delta_0(\Omega)]$. Let us now set
\begin{equation*}
	\delta:=\left(\frac{V^{\frac2n}}{\mu}\right)^{\frac14}.	
\end{equation*}
It follows from the two previous upper bounds that, if $\mu$ is large enough so that
\begin{equation}
\label{eqBoundI}
	\left(\frac{V^{\frac2n}}{\mu}\right)^{\frac14}\le \delta_0(\Omega),
\end{equation}
then we have the following upper bound for the nodal count
\begin{multline}
\label{eqUpperNodalDimN}
\nu(u)\le \frac1{\Lambda(n)^{\frac{n}2}}\left(\left(\frac{1+\eps_0}{1-\eps_0}\xi+\frac{1+\frac{1}{\eps_0}}{1-\eps_0}C^2\xi^{1/2}\right)^{\frac{n}2}+\right.\\
\left.\frac{3\cdot2^{\frac{n}2-1}MK^{\frac{n}2}\rho^2}{m_-\eps_0^{\frac{n}2}}{\xi^{-\frac14}}
\left(\xi+C^2\xi^{1/2}\right)^{\frac{n}2}
\right),
\end{multline}
where $\xi:=V^{2/n}\mu$ and $\rho:=S^{1/2}/V^{1/2-1/2n}$.

An upper bound for the remainder $R_\Omega^D(\mu)$ is given in Section 2 of \cite{vdBkG16cs} for arbitrary dimension. Repeating the argument in Section \ref{secRemainder}, we obtain, for any $\ell\in(0,3(4\sqrt{n})^{-1}\delta_0(\Omega))$,
\begin{equation*}
	R_\Omega^D(\mu)\le \frac{\omega_nMS \sqrt{n}\ell}{(2\pi)^n}\mu^{\frac{n}2}+\frac{\pi n^{\frac32}\omega_nV}{(2\pi)^n\ell}\mu^{\frac{n-1}2}.
\end{equation*}
Setting
\begin{equation*}
	\ell=\sqrt{\frac{\pi n V}{M S}}\mu^{-\frac14},
\end{equation*}
we find that, if $\mu$ is large enough so that
\begin{equation}
\label{eqBoundII}
	\sqrt{\frac{\pi n V}{M S}}\mu^{-\frac14}\le \frac{3\delta_0(\Omega)}{4\sqrt n},
\end{equation}
we have the following upper bound for the remainder
\begin{equation}
\label{eqRemainderDimN}
R_\Omega^D(\mu)\le \frac{2n\omega_n}{(2\pi)^n}\sqrt{\pi M }\rho \xi^{\frac{n}2-\frac14}.
\end{equation}

As before, if the eigenfunction $u$ is Courant-sharp,
\begin{equation*}
	\nu(u)=N_\Omega(\mu)+1,
\end{equation*}
and therefore, if Inequalities \eqref{eqBoundI} and \eqref{eqBoundII} are satisfied, from Inequalities \eqref{eqUpperNodalDimN} and \eqref{eqRemainderDimN}, we obtain
\begin{multline}
\label{eqIneqTestNDim}
	f_\rho(\xi):=\frac{\omega_n}{(2\pi)^n}\xi^{\frac{n}2}-\frac{2n\omega_n\sqrt{\pi M}}{(2\pi)^n}\rho \xi^{\frac{n}2-\frac14}-\frac1{\Lambda(n)^\frac{n}2}\left(\frac{1+\eps_0}{1-\eps_0}\xi+\frac{1+\frac{1}{\eps_0}}{1-\eps_0}{C^2}\xi^{\frac12}\right)^{\frac{n}2}-\\
	\frac{3\cdot2^{\frac{n}2-1}MK^\frac{n}2}{\Lambda(n)^\frac{n}2 m_-\eps_0^\frac{n}2}\rho^2\xi^{-\frac14}\left(\xi+C^2\xi^\frac12\right)^\frac{n}2<0.
\end{multline}
A straightforward computation shows that, since
\begin{equation*}
	0<\eps_0<\frac{1-\gamma(n)^\frac2n}{1+\gamma(n)^\frac2n},
\end{equation*}
we have
\begin{equation*}
	\lim_{\xi\to+\infty}{f_\rho}(\xi)=+\infty.
\end{equation*}
Taking into account the conditions \eqref{eqBoundI}, \eqref{eqBoundII} and \eqref{eqIneqTestNDim}, we can summarise our result in the following proposition.

\begin{prop} \label{propBoundGenNDim}
Let $\Omega$ be an open, bounded and connected set with a $C^2$  boundary. If $\mu$ is a Courant-sharp eigenvalue of the Neumann Laplacian  on $\Omega$ or of $-\Delta_{\Omega}^{\beta}$, we have
\begin{equation*}
\mu\le \max\left(\left(\frac{|\Omega|^\frac1{2n}}{\delta_0(\Omega)}\right)^4,\,\left(\sqrt{\frac\pi{M}}\frac{4n|\Omega|^\frac1{2n}}{3\rho(\Omega)\delta_0(\Omega)}\right)^4,\,
\frac{\xi^*(\rho(\Omega))}{|\Omega|^{\frac2{n}}}\right),
\end{equation*}
with
\begin{equation*}
	\xi^*(\rho):=\sup\{\xi>0\,;\,f_\rho(\xi)<0\},
\end{equation*}
the function $t\mapsto f_\rho(t)$ being defined in Formula \eqref{eqIneqTestNDim}.
\end{prop}
We note that the function $f_\rho$, and therefore also $\xi^*(\rho)$, only depend on the dimension $n$, the geometric properties of the domain $\Omega$ and the choice of $\eps_0$.

\subsection{Geometric upper bounds for Courant-sharp eigenvalues of convex bodies}\label{Sec9.1}
We now assume that the set $\Omega$ is convex. In that case, $t_+(\Omega)=\underline\delta_+(\Omega)$, as  discussed in Remark~\ref{remconvex}, and therefore $\delta_0(\Omega)=t_+(\Omega)$. To simplify the formulas, we write $t_+:=t_+(\Omega)$. Again $\eps_0$ is fixed as in \eqref{eqEpsNDim}. Additionally, the constants appearing on the left-hand side of Inequality \eqref{eqIneqTestNDim} can be chosen as in the second row of Table \ref{tabConstndim}, which implies a smaller value of $\xi^*(\rho)$. Furthermore, again from Remark \ref{remconvex}, Corollary \ref{corVolTube} holds, with $M=1$ and for any $r>0$. The estimate \eqref{eqRemainderDimN} on the remainder therefore holds without Condition \eqref{eqBoundII}. We conclude that, if $\mu$ is Courant-sharp, then
\begin{equation*}
	\mu\le \max\left(\left(\frac{V^\frac1{2n}}{t_+}\right)^4,\,\frac{\xi^*(\rho)}{V^{\frac2{n}}}\right).
\end{equation*}

We now obtain an explicit upper bound for $\xi^*(\rho)$.

From Inequality \eqref{eqIneqTestNDim} and the constants from the second row of Table~\ref{tabConstndim},
we have
\begin{multline}
\label{eqIneqTestNDimconvex}
	\frac{\omega_n}{(2\pi)^n}\xi^{\frac{n}2}-\frac{2n\omega_n\sqrt{\pi}}{(2\pi)^n}\rho \xi^{\frac{n}2-\frac14}-\frac1{\Lambda(n)^\frac{n}2}\left(\frac{1+\eps_0}{1-\eps_0}\xi+12\frac{1+\frac{1}{\eps_0}}{1-\eps_0} \xi^{\frac12}\right)^{\frac{n}2}-\\
	\frac{3\cdot2^{\frac{n}2-1}(4^{n-1})^{\frac{n}2 +1}}{\Lambda(n)^\frac{n}2 \eps_0^\frac{n}2}\rho^2\xi^{-\frac14}\left(\xi+12\xi^\frac12\right)^\frac{n}2<0.
\end{multline}

In what follows, we use the fact that, by \eqref{eqBoundI},
\begin{equation*}
  \xi = V^{2/n}\mu \geq \left(\frac{V^{1/n}}{t_{+}}\right)^4 \geq \omega_n^{4/n},
\end{equation*}
since $V = \vert \Omega \vert \geq \vert B_{t_{+}} \vert = \omega_n t_{+}^n$ as $B_{t_{+}} \subset \Omega$.

We also invoke the Mean Value theorem which gives that for $N \geq 1$, there exists $z \in (0,x)$
such that
\begin{equation*}
  (1+x)^N -1 = N(1+z)^{N-1}x.
\end{equation*}

We have that
\begin{align*}
&\frac1{\Lambda(n)^\frac{n}2}\left(\frac{1+\eps_0}{1-\eps_0}\xi+12\frac{1+\frac{1}{\eps_0}}{1-\eps_0} \xi^{\frac12}\right)^{\frac{n}2}\notag\\
&= \frac1{\Lambda(n)^\frac{n}2}\left(\frac{1+\eps_0}{1-\eps_0}\right)^{\frac{n}2}\xi^{\frac{n}2}
\left(1+12\frac{1+\frac{1}{\eps_0}}{1+\eps_0} \xi^{-\frac12}\right)^{\frac{n}2}\notag\\
&\leq \frac1{\Lambda(n)^\frac{n}2}\left(\frac{1+\eps_0}{1-\eps_0}\right)^{\frac{n}2}\xi^{\frac{n}2}
\left(1 + 6n \left(1+12\frac{1+\frac{1}{\eps_0}}{1+\eps_0} \xi^{-\frac12}\right)^{\frac{n}2 -1}\frac{1+\frac{1}{\eps_0}}{1+\eps_0} \xi^{-\frac12}\right)\notag\\
&\leq \frac1{\Lambda(n)^\frac{n}2}\left(\frac{1+\eps_0}{1-\eps_0}\right)^{\frac{n}2}\xi^{\frac{n}2}
\left(1 + 6n \left(1+12\frac{1+\frac{1}{\eps_0}}{1+\eps_0} \omega_n^{-\frac2{n}}\right)^{\frac{n}2 -1}\frac{1+\frac{1}{\eps_0}}{1+\eps_0} \xi^{-\frac12}\right)\notag\\
&= \frac1{\Lambda(n)^\frac{n}2}\left(\frac{1+\eps_0}{1-\eps_0}\right)^{\frac{n}2}\xi^{\frac{n}2} \notag\\
& \ \ \ + \frac{6n}{\Lambda(n)^\frac{n}2}\left(\frac{1+\eps_0}{1-\eps_0}\right)^{\frac{n}2}
\left(\frac{1+\frac{1}{\eps_0}}{1+\eps_0}\right)\left(1+12 \left(\frac{1+\frac{1}{\eps_0}}{1+\eps_0}\right) \omega_n^{-\frac2{n}}\right)^{\frac{n}2-1} \xi^{\frac{n-1}2},
\end{align*}
and
\begin{equation*}
  (\xi+12\xi^{\frac12})^{\frac{n}2} = \xi^{\frac{n}2}(1+12\xi^{-\frac12})^{\frac{n}2}
  \leq \xi^{\frac{n}2}(1+12 \omega_n^{-\frac2{n}})^{\frac{n}2}.
\end{equation*}

Substituting these into \eqref{eqIneqTestNDimconvex}, we obtain
\begin{multline}
\label{eqIneqTestNDimconvex2}
	\left(\frac{\omega_n}{(2\pi)^n} - \frac1{\Lambda(n)^\frac{n}2}\left(\frac{1+\eps_0}{1-\eps_0}\right)^{\frac{n}2}\right) \xi^{\frac{n}2}\\ < \left(\frac{2n\omega_n\sqrt{\pi}}{(2\pi)^n}\rho + \frac{3\cdot2^{\frac{n}2-1}(4^{n-1})^{\frac{n}2 +1}}{\Lambda(n)^\frac{n}2 \eps_0^\frac{n}2}(1+12 \omega_n^{-\frac2{n}})^{\frac{n}2}\rho^2 \right) \xi^{\frac{n}2-\frac14} \\ + \frac{6n}{\Lambda(n)^\frac{n}2}\left(\frac{1+\eps_0}{1-\eps_0}\right)^{\frac{n}2}
\left(\frac{1+\frac{1}{\eps_0}}{1+\eps_0}\right)\left(1+12 \left(\frac{1+\frac{1}{\eps_0}}{1+\eps_0}\right) \omega_n^{-\frac2{n}}\right)^{\frac{n}2-1} \xi^{\frac{n-1}2},
\end{multline}
We note that the coefficients of $\xi^{\frac{n}2}, \xi^{\frac{n-1}2}$ only depend on the dimension,
while the coefficient of $\xi^{\frac{n}2-\frac14}$ depends on the dimension and the isoperimetric ratio.

Inequality \eqref{eqIneqTestNDimconvex2} implies that
\begin{multline}
\label{eqIneqTestNDimconvex3}
\xi < D_n^{-1}\left(\frac{2n\omega_n\sqrt{\pi}}{(2\pi)^n}\rho + \frac{3\cdot2^{\frac{n}2-1}(4^{n-1})^{\frac{n}2 +1}}{\Lambda(n)^\frac{n}2 \eps_0^\frac{n}2}(1+12 \omega_n^{-\frac2{n}})^{\frac{n}2}\rho^2 \right) \xi^{\frac34} \\
+ \frac{6n}{\Lambda(n)^\frac{n}2}\left(\frac{1+\eps_0}{1-\eps_0}\right)^{\frac{n}2}
\left(\frac{1+\frac{1}{\eps_0}}{1+\eps_0}\right)\left(1+12 \left(\frac{1+\frac{1}{\eps_0}}{1+\eps_0}\right) \omega_n^{-\frac2{n}}\right)^{\frac{n}2-1} \xi^{\frac34}\xi^{-\frac14},
\end{multline}
where $D_n:=\frac{\omega_n}{(2\pi)^n} - \frac1{\Lambda(n)^\frac{n}2}\left(\frac{1+\eps_0}{1-\eps_0}\right)^{\frac{n}2}$.
Substituting $\xi^{-\frac14} \leq \omega_n^{-\frac1{n}}$ into Inequality \eqref{eqIneqTestNDimconvex3}, we obtain
\begin{align}
\label{eqIneqTestNDimconvex4}
&\xi < D_n^{-4}\Biggl(\frac{2n\omega_n\sqrt{\pi}}{(2\pi)^n}\rho + \frac{3\cdot2^{\frac{n}2-1}(4^{n-1})^{\frac{n}2 +1}}{\Lambda(n)^\frac{n}2 \eps_0^\frac{n}2}(1+12 \omega_n^{-\frac2{n}})^{\frac{n}2}\rho^2 \notag \\
& \ \ \ \ \ \ \ + \frac{6n \omega_{n}^{-\frac1{n}}}{\Lambda(n)^\frac{n}2}\left(\frac{1+\eps_0}{1-\eps_0}\right)^{\frac{n}2}
\left(\frac{1+\frac{1}{\eps_0}}{1+\eps_0}\right)\left(1+12 \left(\frac{1+\frac{1}{\eps_0}}{1+\eps_0}\right) \omega_n^{-\frac2{n}}\right)^{\frac{n}2-1}\Biggr)^{4},
\end{align}
which implies that
\begin{align}
\label{eqIneqTestNDimconvex5}
&\mu < D_n^{-4}V^{-2/n}\Biggl(\frac{2n\omega_n\sqrt{\pi}}{(2\pi)^n}\rho + \frac{3\cdot2^{\frac{n}2-1}(4^{n-1})^{\frac{n}2 +1}}{\Lambda(n)^\frac{n}2 \eps_0^\frac{n}2}(1+12 \omega_n^{-\frac2{n}})^{\frac{n}2}\rho^2 \notag \\
& \ \ \ \ \ \ \ + \frac{6n \omega_{n}^{-\frac1{n}}}{\Lambda(n)^\frac{n}2}\left(\frac{1+\eps_0}{1-\eps_0}\right)^{\frac{n}2}
\left(\frac{1+\frac{1}{\eps_0}}{1+\eps_0}\right)\left(1+12 \left(\frac{1+\frac{1}{\eps_0}}{1+\eps_0}\right) \omega_n^{-\frac2{n}}\right)^{\frac{n}2-1}\Biggr)^{4} \notag\\
& \ \ \ \ \ \ \ \ =: M_1(n,\rho,V).
\end{align}

If instead we consider the bound due to Payne and Weinberger, \cite{PW},
\begin{equation}
\label{eqPWDimN}
  \mu_2(\Omega) \geq \frac{\pi^2}{\diam(\Omega)^2},
\end{equation}
then
\begin{equation*}
  \xi^{-\frac14} \leq \frac{\diam(\Omega)^{\frac12}}{\sqrt{\pi}V^{\frac1{2n}}}.
\end{equation*}

The following bound due to Gritzmann, Wills and Wrase, \cite{GWW},
\begin{equation*}
  \diam(\Omega) < \frac{S^{n-1}}{\omega_{n-1}(nV)^{n-2}},
\end{equation*}
gives that
\begin{equation*}
  \xi^{-\frac14} \leq \frac{\rho^{n-1}}{\sqrt{\pi} \omega_{n-1}^{\frac12}n^{\frac{n}2-1}}.
\end{equation*}
By substituting the latter inequality into Inequality \eqref{eqIneqTestNDimconvex3}, we have
\begin{align}
\label{eqIneqTestNDimconvex6}
&\mu < D_n^{-4}V^{-\frac2n}\Biggl(\frac{2n\omega_n\sqrt{\pi}}{(2\pi)^n}\rho + \frac{3\cdot2^{\frac{n}2-1}(4^{n-1})^{\frac{n}2 +1}}{\Lambda(n)^\frac{n}2 \eps_0^\frac{n}2}(1+12 \omega_n^{-\frac2{n}})^{\frac{n}2}\rho^2 \notag \\
& + \frac{6 \rho^{n-1}}{\sqrt{\pi} \omega_{n-1}^{\frac12} n^{\frac{n}2-1}\Lambda(n)^\frac{n}2}\left(\frac{1+\eps_0}{1-\eps_0}\right)^{\frac{n}2}
\left(\frac{1+\frac{1}{\eps_0}}{1+\eps_0}\right)\left(1+12 \left(\frac{1+\frac{1}{\eps_0}}{1+\eps_0}\right) \omega_n^{-\frac2{n}}\right)^{\frac{n}2-1}\Biggr)^{4}\notag\\
&=:M_2(n,\rho,V)
\end{align}

We can summarise both versions of the upper bound in the following proposition.

\begin{prop} \label{propConvNDim} Let $\Omega$ be an open, bounded and convex set with a $C^2$ boundary. If $\mu$ is a Courant-sharp eigenvalue of the Neumann Laplacian on $\Omega$ or of $-\Delta_{\Omega}^{\beta}$, we have
\begin{equation}
\label{eqUBmuND1}
	\mu\le \max\left(\frac{|\Omega|^\frac2n}{t_+(\Omega)^4},\,M_1(n,\rho(\Omega), |\Omega|)\right)
\end{equation}
and
\begin{equation*}
	\mu\le \max\left(\frac{|\Omega|^\frac2n}{t_+(\Omega)^4},\,M_2(n,\rho(\Omega), |\Omega|)\right),
\end{equation*}
with $M_1$ and $M_2$ defined in Formulas \eqref{eqIneqTestNDimconvex5} and \eqref{eqIneqTestNDimconvex6} respectively.
\end{prop}

Similarly to the end of Section \ref{Sec8}, we can find simpler, although less explicit, upper bounds.
\begin{prop} \label{propConvNDimSimple} Let $\Omega$ be an open, bounded and convex set with a $C^2$ boundary. Let  $\mu=\mu_k(\Omega)$ or $\mu=\mu_k(\Omega,\beta)$, with $\beta$ a non-negative Lipschitz continuous function on $\partial\Omega$. There exist positive constants $C_n$ and $C_n'$ that do not depend on $\Omega$ such that,
if $\mu$ is Courant-sharp,
\begin{equation*}
	\mu\le C_n\left(\frac{|\Omega|^\frac2n}{t_+(\Omega)^4}+\frac{\rho(\Omega)^8}{|\Omega|^\frac2n}\right)
\end{equation*}
and
\begin{equation*}
	k\le C_n'\left(\left(\frac{|\Omega|^\frac1{n}}{t_+(\Omega)}\right)^{2n}+\rho(\Omega)^{4n}\right).
\end{equation*}
\end{prop}

The first inequality in Proposition \ref{propConvNDimSimple} follows from Inequalities \eqref{eqIneqTestNDimconvex4} and \eqref{eqUBmuND1}, taking into account $\rho(\Omega)\ge \sqrt n \omega_n^\frac1{2n}$.  The second follows from the first and  Corollary \ref{corUBk}, together with Remark \ref{remAppRobin}, after applying Young's inequality repeatedly.  Both $C_n$ and $C_n'$ could be computed explicitly.

As described in the introduction for the case of dimension 2, the previous inequality implies that a sufficiently regular convex set with a large number of Courant-sharp eigenvalues has a large isoperimetric ratio or a point in its boundary where the curvature is large (or both). Additionally, if there exists a Courant-sharp eigenvalue which is large with respect to $\vert\Omega\vert^\frac2nt_+(\Omega)^{-4}$, the set has a large isoperimetric ratio.

\textbf{Acknowledgments:}
KG was supported by the Swiss National Science Foundation grant no.\ 200021\_163228 entitled \emph{Geometric Spectral Theory} during the academic year 2016 - 2017.
KG also acknowledges support from the Max Planck Institute for Mathematics, Bonn, from 1 October 2017 to 31 July 2018.
CL acknowledges support from the ERC (project  COMPAT,  ERC-2013-ADG no. 339958), the  Portuguese FCT (Project OPTFORMA, IF/00177/2013) and the Swedish Research Council (Grant D0497301).
Both authors wish to thank the referee for her/his insightful comments and suggestions.

\appendix
\section{An upper bound for the Neumann counting function of a convex domain}\label{sec:A}

\begin{prop} \label{propUBN} Let $\Omega$ be an open, bounded and convex set in $\RR^n$ with a $C^2$ boundary. Then, for all $\mu>0$, we have
\begin{equation}
\label{eqIneqUBN}
N^N_\Omega(\mu)\le \frac{n^\frac{n}2}{\pi^n}|\Omega|\mu^\frac{n}2+\frac{n^\frac{n}2}{\pi^n}|\partial\Omega|\sum_{i=0}^{n-1}\frac{\left(\begin{array}{c}n-1\\i\end{array}\right)\pi^{i+1}}{(i+1)t_+(\Omega)^i}\mu^{\frac{n-i-1}2},
\end{equation}
where $t_+(\Omega)$ is the smallest radius of curvature, as defined by Equation \eqref{eqCriticalt}.
\end{prop}

\begin{proof} We combine the Payne-Weinberger lower bound (see Inequality \eqref{eqPWDimN}) with Neumann bracketing. Reference \cite{vdBkG16cs} uses a similar approach with Dirichlet bracketing.

For $\alpha\in \ZZ^n$, we define the open hypercube $C_\alpha:=\alpha+(0,1)^n$. Let us fix $a>0$. We define
\begin{equation*}
	\mathcal I_a:=\left\{\alpha\in\ZZ^n\,;\,aC_\alpha\cap \Omega\neq\emptyset\right\}
\end{equation*}
and denote by $K(a)$ the cardinality of $\mathcal I_a$. For all $\alpha\in \mathcal I_a$, we define $\Omega_{\alpha,a}:=\Omega\cap aC_\alpha$. By construction, $\Omega_{\alpha,a}$ is a convex, open set with diameter bounded from above by
$\sqrt{n}a$. We define
\begin{equation*}
	\widetilde{\Omega}_a:=\bigcup_{\alpha\in\mathcal I_a}\Omega_{\alpha,a},
\end{equation*}
which is open and bounded, but not connected.

We consider the spectrum of the Neumann Laplacian on $\widetilde{\Omega}_a$, which we denote by $(\mu_k(a))_{k\ge1}$. Since $\widetilde{\Omega}_a$ has $K(a)$ connected components, $0$ is an eigenvalue of multiplicity $K(a)$, that is
\begin{equation*}
	0=\mu_1(a)=\mu_2(a)=\dots=\mu_{K(a)}(a).
\end{equation*}
The eigenvalue $\mu_{K(a)+1}(a)$ is the smallest among the non-trivial eigenvalues of the sets $\Omega_{\alpha,a}$ with $\alpha\in \mathcal I_a$, and therefore, according to the Payne-Weinberger inequality,
\begin{equation*}
	\frac{\pi^2}{na^2}\le \mu_{K(a)+1}(a).
\end{equation*}
It follows from the variational characterisation of the eigenvalues that $\mu_k(a)\le \mu_k(\Omega)$ for any integer $k$
(see, for example, Proposition 4(c) of \cite{RS}), so that
\begin{equation*}
	\frac{\pi^2}{na^2}\le \mu_{K(a)+1}(\Omega).
\end{equation*}
Equivalently, the counting function $N^N_\Omega (\cdot)$ satisfies
\begin{equation}
\label{eqIneqNK}
	N^N_\Omega\left(\frac{\pi^2}{na^2}\right)\le K(a).
\end{equation}
To finish the proof, let us give an upper bound for $K(a)$. We define $D_a$ as the interior of
\begin{equation*}
	\overline{\bigcup_{\alpha\in\mathcal I_a}C_{\alpha,a}},
\end{equation*}
and note that $|D_a|=K(a)a^n$. Furthermore, we have
\begin{equation*}
	D_a\subset \Omega+\sqrt{n}a B,
\end{equation*}
where $B$ is the ball of radius $1$ in $\RR^n$ and the right-hand side is understood as a Minkowski sum. We obtain
\begin{equation}
\label{eqIneqKVol}
	K(a)\le a^{-n}\left|\Omega+\sqrt{n}aB\right|.
\end{equation}
Using normal coordinates in the exterior of $\Omega$ (similarly to Subsection~\ref{SubsecTubesNdim} but for the outer parallel sets), we find, for any $\delta>0$,
\begin{equation*}
	\left|\Omega+\delta B\right|=|\Omega|+\int_0^{\delta}\left(\int_{\partial \Omega}\prod_{i=1}^{n-1}(1+t\kappa_i(x'))\,dx'\right)\,dt.
\end{equation*}
Expanding and using $\kappa_i(x')\le t_+(\Omega)^{-1}$ for all $x'\in \partial \Omega$ and $i\in\{1,\dots,n-1\}$, we find
\begin{equation*}
	\left|\Omega+\delta B\right|\le |\Omega|+|\partial\Omega|\sum_{i=0}^{n-1}\frac{\left(\begin{array}{c}n-1\\i\end{array}\right)\delta^{i+1}}{(i+1)t_+^i(\Omega)}.
\end{equation*}
We now choose $\delta:=\sqrt{n}a:=\pi/\sqrt{\mu}$. We obtain the desired result by substituting this value into the previous inequality and by using Inequalities \eqref{eqIneqKVol} and \eqref{eqIneqNK} successively.
\end{proof}

Let us denote by $F_n(|\Omega|,|\partial\Omega|,t_+(\Omega),\mu)$ the right-hand side of Inequality \eqref{eqIneqUBN}. It is a continuous and increasing function of $\mu$.

\begin{cor} \label{corUBk} Let $\Omega$ be an open, bounded and convex set in $\RR^n$ with a $C^2$ boundary. For any integer $k\ge1$,
\begin{equation}
\label{eqIneqUBk}
k\le F_n(|\Omega|,|\partial\Omega|,t_+(\Omega),\mu_k(\Omega)).
\end{equation}
\end{cor}

\begin{proof} Let us fix $\eps>0$. We have $N^N_\Omega((1+\eps)\mu_k(\Omega))\ge k$. We obtain Inequality \eqref{eqIneqUBk} by applying Inequality \eqref{eqIneqUBN} and letting $\eps$ go to $0$.
\end{proof}

\begin{rem} \label{remAppRobin} We recall that $\mu_k(\Omega)\le \mu_k(\Omega,\beta)$ for every non-negative Lipschitz continuous function $\beta$ on $\partial \Omega$, and every $k\ge 1$. Since $F_n(|\Omega|,|\partial\Omega|,t_+(\Omega),\mu)$ is increasing in $\mu$, Inequalities \eqref{eqIneqUBN} and \eqref{eqIneqUBk}  hold when substituting $N^\beta_\Omega(\mu)$ for $N^N_\Omega(\mu)$ and $\mu_k(\Omega,\beta)$ for $\mu_k(\Omega)$, respectively.
\end{rem}

\bibliographystyle{plain}

\section{Corrigendum}
\label{app:corr}

\begin{abstract}
In the above work, the number of nodal domains for the eigenfunctions of the Neumann or Robin Laplacian is bounded from above using the classical (Euclidean) Faber-Krahn inequality. However, in an arbitrary Riemannian manifold, this inequality might not hold. We supply the missing arguments in two dimensions and outline a modification of the method, which preserves most of the results, in $n$ dimensions.
\end{abstract}

\vspace{10pt}

The mapping $F$ introduced in Equation \eqref{eqF} on page 6 should instead be defined from the cylinder $\mathcal C_L:=(\mathbb R/(L\mathbb Z))\times \mathbb R$ to $\mathbb R^2$. The set $V^R$, defined at the bottom of page 10 through a preimage by $F$, is therefore an open set in $\mathcal C_L$ and not in $\mathbb R^2$. It is not a priori clear that the Faber-Krahn inequality with the same constant as in the Euclidean case (hereafter called \emph{classical Faber-Krahn inequality}) holds for $V^R$. The same gap occurs in Section \ref{ssecgeom} when considering a possibly multiply connected domain in $\mathbb R^2$, for which $V^R$ is an open set in the disjoint union $\bigcup_{i=0}^b\mathcal C_{L_i}$. Similarly, when $\Omega\subset\mathbb R^n$, we implicitly apply (on page 23) the classical Faber-Krahn inequality in the cylindrical manifold $N(\Gamma)=\Gamma\times \mathbb R$, where it might not hold.

When $\Omega$ is a (possibly multiply connected) $C^2$ domain in $\mathbb R^2$, we can fix the issue by combining a few known geometric results. Examination of the proof of the classical Faber-Krahn inequality by symmetrization (see for instance \cite[I.9]{BM82-1}) shows that it holds provided all open sets $W\subset V^R$ satisfy the classical isoperimetric inequality $|\partial W|^2\ge 4\pi|W|$. From \cite[\S 6]{ HHM99}, the classical isoperimetric inequality holds for any open set $W_i\subset \mathcal C_{L_i}$ such that $|W_i|\le L_i^2/\pi$. Note that the parameter $a$ in Reference \cite{HHM99} is the radius of the cylinder, equal to $L_i/2\pi$ in our notation. We then remark that $V^R$ can be written as a disjoint union $\bigcup_{i=0}^b V_i$, with each $V_i$ contained in 
\[(\mathbb R/(L_i\mathbb Z))\times (-\delta_0(\Omega),\delta_0(\Omega))\subset\mathcal C_{L_i}\]
(some $V_i$'s may be empty). In addition, $\delta_0(\Omega)$ is less than or equal to the smallest radius of curvature of $\partial \Omega$, denoted by $t_+(\Omega)$. We can write $\Omega=\Omega_0\setminus\bigcup_{i=1}^b \overline{D_i}$, where $\Omega_0$ and all the $D_i$'s are simply connected $C^2$ domains and apply to each the inequality $\pi t_+(D)^2\le |D|$, valid for any simply connected domain $D \subset \mathbb R^2$ (see \cite{P2015} and references therein). We obtain 
\[\delta_0(\Omega)\le t_+(\Omega)\le \frac{1}{\pi^{1/2}}\min\left\{|\Omega_0|^{1/2},\min_{1\le i\le b}|D_i|^{1/2}\right\}\le \frac{1}{2\pi}\min_{0\le i\le b}L_i\]
(where the last inequality follows from the isoperimetric inequality in $\mathbb R^2$). Thus, $|V_i|\le 2\,\delta_0(\Omega)\,L_i\le  L_i^2/\pi$ for all $i$, and any $W_i\subset V_i$ satisfies the isoperimetric inequality. This is therefore also the case for any $W=\bigcup_{i=0}^b W_i\subset V^R$, which concludes the proof. All the results of Sections 2 to 8 hold without any change.

 Most of the material in Section \ref{SecNdim} can be recovered by replacing the quantities $\delta_+(x')$, $\underline{\delta}_+(\Omega)$ and $\delta_0(\Omega)$, defined on page 21, with respectively
\begin{align*}
	\delta(x')&:=\sup\{\delta>0\,:\,\mbox{dist}(F(x',t),\partial \Omega)=|t|\mbox{ for all }t\in[-\delta,\delta]\},\\
	\underline{\delta}(\Omega)&:=\inf_{x'\in\partial \Omega}\delta(x'),\\
	\delta_1(\Omega)&:=\min\{t_+(\Omega),\underline{\delta}(\Omega)\}.
\end{align*} 
Having defined the set $V^R\subset \Gamma \times \mathbb R$ as in Section \ref{ssecNoofnodaldomains}, we modify our method by taking the additional steps of considering the image $U^R:=F(V^R)\subset \mathbb R^n$ and applying the Faber-Krahn inequality to $U^R$. We get an upper bound on the number of boundary nodal domains $\nu_1(\varepsilon_0,u)$, and thus find results similar to Proposition \ref{propBoundGenNDim}, except that $\delta_0(\Omega)$ is replaced with $\delta_1(\Omega)$ and the constants depending only on $n$ are larger ones. Note that the geometric quantity $\delta_1(\Omega)$  depends on the cut-distance to the boundary (also known as the injectivity radius of the normal exponential map) with respect to both the interior and the exterior of the domain $\Omega$, whereas $\delta_0(\Omega)$ depends only on the interior. If, in addition, $\Omega$ is assumed to be convex, then $\delta_1(\Omega)=\delta_0(\Omega)$, since $\mbox{dist}(F(x',t),\partial \Omega)=-t$ for all $x'\in\partial \Omega$ and $t\le 0$ (the cut-distance with respect to the exterior is infinite). Therefore, in that case, only the constants depending on $n$ change. In particular, Proposition \ref{propConvNDimSimple} holds as written.

\bibliographystyle{plain}

\end{document}